\documentclass[11pt,oneside,english,final]{amsart}
\usepackage[T1]{fontenc}
\usepackage[latin9]{inputenc}
\usepackage{geometry}
\usepackage[active]{srcltx}
\usepackage{refstyle}
\usepackage{float}
\usepackage{url}
\usepackage{amsthm}
\usepackage{amstext}
\usepackage{amssymb}
\usepackage{graphicx}
\usepackage{setspace}
\usepackage{esint}
\makeatletter

\AtBeginDocument{\providecommand\secref[1]{\ref{sec:#1}}}
\AtBeginDocument{\providecommand\thmref[1]{\ref{thm:#1}}}
\AtBeginDocument{\providecommand\eqref[1]{\ref{eq:#1}}}
\AtBeginDocument{\providecommand\remref[1]{\ref{rem:#1}}}
\AtBeginDocument{\providecommand\lemref[1]{\ref{lem:#1}}}
\AtBeginDocument{\providecommand\factref[1]{\ref{fact:#1}}}
\providecommand{\tabularnewline}{\\}
\RS@ifundefined{subref}
  {\def\RSsubtxt{section~}\newref{sub}{name = \RSsubtxt}}
  {}
\RS@ifundefined{thmref}
  {\def\RSthmtxt{theorem~}\newref{thm}{name = \RSthmtxt}}
  {}
\RS@ifundefined{lemref}
  {\def\RSlemtxt{lemma~}\newref{lem}{name = \RSlemtxt}}
  {}

\numberwithin{equation}{section}
\numberwithin{figure}{section}
\theoremstyle{plain}
\newtheorem{thm}{\protect\theoremname}
  \theoremstyle{remark}
  \newtheorem{rem}[thm]{\protect\remarkname}
  \theoremstyle{plain}
  \newtheorem{fact}[thm]{\protect\factname}
  \theoremstyle{plain}
  \newtheorem{lem}[thm]{\protect\lemmaname}

\usepackage{color}

\def\RSsubtxt{Section~}\newref{sub}{name = \RSsubtxt}
\def\RSsubtxt{Section~}\newref{sec}{name = \RSsubtxt}
\def\RSthmtxt{Theorem~}\newref{thm}{name = \RSthmtxt}
\def\RSlemtxt{Lemma~}\newref{lem}{name = \RSlemtxt}

\usepackage[unicode=true,
 bookmarks=false,
 breaklinks=false,pdfborder={0 0 1},backref=false,colorlinks=false]
 {hyperref}

\makeatother

\usepackage{babel}
  \providecommand{\factname}{Fact}
  \providecommand{\lemmaname}{Lemma}
  \providecommand{\remarkname}{Remark}
\providecommand{\theoremname}{Theorem}

\begin{document}
\global\long\def\intr{\int_{R}}
 \global\long\def\sbr#1{\left[ #1\right] }
\global\long\def\cbr#1{\left\{  #1\right\}  }
\global\long\def\rbr#1{\left(#1\right)}
\global\long\def\ev#1{\mathbb{E}{#1}}
\global\long\def\R{\mathbb{R}}
\global\long\def\norm#1#2#3{\Vert#1\Vert_{#2}^{#3}}
\global\long\def\pr#1{\mathbb{P}\rbr{#1}}
\global\long\def\cleq{\lesssim}
\global\long\def\ceq{\eqsim}
\global\long\def\conv{\rightarrow}
\global\long\def\Var#1{\text{Var}(#1)}
\global\long\def\TDD#1{{\color{red}To\, Do(#1)}}
\global\long\def\dd#1{\textnormal{d}#1}
\global\long\def\inti{\int_{0}^{\infty}}
\global\long\def\crr{\mathcal{C}([0;\infty),\R)}
\global\long\def\sb#1{\langle#1\rangle}
\global\long\def\pm#1{d_{P}\rbr{#1}}
\global\long\def\crt{\mathcal{C}([0;T],\R)}
\global\long\def\nuu{\nu_{n;\lambda}}
\global\long\def\ZZ{Z_{\Lambda_{n}}}
\global\long\def\PP{\mathbb{P}_{\Lambda_{n}}}
\global\long\def\EE{\mathbb{E}_{\Lambda_{n}}}
\global\long\def\LL{\Lambda_{n}}
\global\long\def\AA{\mathcal{A}}
\global\long\def\evx{\mathbb{E}_{x}}
\global\long\def\pin#1{1_{\cbr{#1\in\mathcal{A}}}}
\global\long\def\Zd{\mathbb{Z}^{d}}
\global\long\def\TT{T_{n}}
\global\long\def\ZZa{Z_{n;\epsilon}}
\global\long\def\emax{\bar{\epsilon}}
\global\long\def\emin{\underbar{\ensuremath{\epsilon}}}
\global\long\def\estar{\epsilon^{*}}
\global\long\def\ee{\mathbf{e}}
\global\long\def\ddp#1#2{\langle#1,#2\rangle}
\global\long\def\intc#1{\int_{0}^{#1}}
\global\long\def\T#1{\mathcal{P}_{#1}}
\global\long\def\ii{\mathbf{i}}
\global\long\def\star#1{\left.#1^{*}\right.}
\global\long\def\pspace{\mathcal{C}}
\global\long\def\eq{\varphi}
\global\long\def\grad{\text{grad}}
\global\long\def\var{\text{Var}}
\global\long\def\ab{[a;b]}
\global\long\def\TTV#1#2#3{\text{TV}^{#3}\!\rbr{#1,#2}}
\global\long\def\UTV#1#2#3{\text{UTV}^{#3}\!\rbr{#1,#2}}
\global\long\def\DTV#1#2#3{\text{DTV}^{#3}\!\rbr{#1,#2}}

\linespread{1.03}
\title{On truncated variation, upward truncated variation and downward truncated
variation for diffusions}

\author{{ Rafa\l{} M. \L{}ochowski $^\dag$ } \\{\tiny  Department of Mathematics and Mathematical Economics,\\Warsaw School of Economics, \\
Madali\'{n}skiego  6/8, 02-513 Warszawa, Poland \\email:rlocho@sgh.waw.pl }\\~\\Piotr Mi\l{}o\'s$^*$ \\{\tiny Faculty of Mathematics, Informatics, and Mechanics,\\Banacha 2, 02-097 Warszawa, Poland\\email:pmilos@mimuw.edu.pl}}
\thanks{$^*$Corresponding author: tel: +48 22-55-44-122, fax:+48 22-55-44-300. The author was partially supported by the National Science Centre under decision no. DEC-2011/01/B/ST1/05089 and by the Ministry of Science grant N N201 397537}
\thanks{$^\dag$This research was supported in part by the National Science Centre under decision no. DEC-2011/01/B/ST1/05089}

\begin{abstract} 
The truncated variation, $\text{TV}^{c}$, is a fairly new concept
introduced in \cite{Lochowski:2008}. Roughly speaking, given a c\`{a}dl\`{a}g
function $f$, its truncated variation is {}``the total variation
which does not pay attention to small changes of $f$, below some
threshold $c>0$''. The very basic consequence of such approach is
that contrary to the total variation, $\text{TV}^{c}$ is always finite.
This is appealing to the stochastic analysis where so-far large classes
of processes, like semimartingales or diffusions, could not be studied
with the total variation. Recently in \cite{ochowski:2011lr}, another
characterization of $\text{TV}^{c}$ was found. Namely $\text{TV}^{c}$
is the smallest possible total variation of a function which approximates $f$
uniformly with accuracy $c/2$. Due to these properties we envisage
that $\text{TV}^{c}$ might be a useful concept both in the theory
and applications of stochastic processes. 

For this reason we decided to determine some properties of $\text{TV}^{c}$
for some well-known processes. In course of our research we discover
intimate connections with already known concepts of the stochastic
processes theory.

Firstly, for semimartingales we proved that $\text{TV}^{c}$ is of
order $c^{-1}$ and the normalized truncated variation converges almost
surely to the quadratic variation of the semimartingale as $c\searrow0$.
Secondly, we studied the rate of this convergence. As this task was
much more demanding we narrowed to the class of diffusions (with some
mild additional assumptions). We obtained the weak convergence to
a so-called Ocone martingale. These results can be viewed as some
kind of law of large numbers and the corresponding central limit
theorem. 

Finally, for a Brownian motion with a drift we proved the behavior
of $\text{TV}^{c}$ on intervals going to infinity. Again, we obtained
a LLN and CLT, though in this case they have a different interpretation
and were easier to prove. 

All the results above were obtained in a functional setting, viz.
we worked with processes describing the growth of the truncated variation
in time. Moreover, in the same respect we also treated two closely
related quantities - the so-called upward truncated variation and
downward truncated variation. 
\end{abstract}

\keywords{Stochastic processes, semimartingales, diffusions, truncated variation,
total variation. }

\maketitle

\section{Introduction and results}

Recently, the following notion of \emph{the truncated variation }has
been introduced in \cite{Lochowski:2008}:

\begin{equation}
\TTV f{\ab}c:=\sup_{n}\sup_{a\leq t_{1}<t_{2}<\ldots<t_{n}\leq b}\sum_{i=1}^{n-1}\phi_{c}\left(\left|f(t_{i+1})-f(t_{i})\right|\right),\label{eq:TVDefinition}
\end{equation}
where $\phi_{c}\left(x\right)=\max\left\{ x-c,0\right\} ,\: c\geq0$
and $f:\ab\mapsto\R$ is a c\`{a}dl\`{a}g function. The trivial observation
is that $\text{TV}^{0}$ is nothing else that the total variation
(which will be also denoted by $\text{TV}$). The introduction of
the truncation parameter $c$ makes it possible to circumvent a classical
problem of stochastic analysis; namely, that the total variation of
the Brownian motion as well as of a `non-trivial` diffusion process
is almost surely infinite. This alone makes $\text{TV}^{c}$ an interesting
research object. Other properties of $\text{TV}^{c}$ were found,
amongst which the variational characterization of the truncated variation
given by
\begin{equation}
\TTV f{\ab}c=\inf\cbr{\TTV g{\ab}{}:g\,\text{such that }\norm{g-f}{\infty}{}\leq\frac{1}{2}c},\label{eq:variational}
\end{equation}
where $\norm g{\infty}{}:=\sup\cbr{|g(x)|:x\in\ab}$. In other words,
truncated variation is the lower bound for the total variation of
functions approximating $f$ with accuracy $c$. It appears that the
$\inf$ in the above expression is attained at some function $g^{c}$.
The properties just listed give hope that $\text{TV}^{c}$ could be
used in the stochastic analysis. This question is a active field of
research, some promising results are contained in \cite{ochowski:2011rs},
like definition of a stochastic integral with respect to a semimartingale
as a limit of the pathwise Riemann-Stieltjes stochastic integrals,
and other are being investigated. A detailed description would be
too vast for our introduction therefore we refer the reader to \cite{ochowski:2011rs} and \cite{ochowski:2011lr}, and its debriefing in \secref{PropertiesTV}. 

Having agreed that $\text{TV}^{c}$ might be a useful tool, an important
task is to describe the behavior of $\text{TV}^{c}$ for a vast class
of stochastic processes. This is the main aim of this paper. We will
derive first order properties for \emph{continuous semimartingales
}and second order properties for\emph{ continuous diffusions }(under
some mild technical assumptions) when $c\searrow0$. Intuitively,
these answer the question of how fast $\text{TV}^{c}$ converges to
the total variation, that is how fast it diverges to infinity. In
the case of the Brownian motion with drift we will also study the
behavior of $\text{TV}^{c}$ on large time intervals. 

Before presenting our results we define two concepts closely related
to $\text{TV}^{c}$. The upward truncated variation given by
\begin{equation}
\UTV f{\ab}c:=\sup_{n}\sup_{a\leq t_{1}<s_{1}<t_{2}<s_{2}<...<t_{n}<s_{n}\leq b}\sum_{i=1}^{n}\phi_{c}\left(f(s_{i})-f(t_{i})\right),\label{eq:defUTV}
\end{equation}
and the downward truncated variation given by 
\[
\DTV f{\ab}c:=\sup_{n}\sup_{a\leq t_{1}<s_{1}<t_{2}<s_{2}<...<t_{n}<s_{n}\leq b}\sum_{i=1}^{n}\phi_{c}\left(f(t_{i})-f(s_{i})\right).
\]
The relation between $\text{TV}^{c},\text{UTV}^{c},\text{DTV}^{c}$
will become clear in Section \ref{sub:Joint-structure-of}. Given
a c�dl�g process $\cbr{X_{t}}_{t\ge0}$ we define the following families
of processes $\cbr{\TTV Xtc}_{t\geq0}$, $\cbr{\UTV Xtc}_{t\geq0}$
and $\cbr{\DTV Xtc}_{t\geq0}$ by 
\[
\TTV Xtc:=\TTV X{[0;t]}c,\quad\UTV Xtc:=\UTV X{[0;t]}c,\quad\DTV Xtc:=\DTV X{[0;t]}c,
\]
where all the above definitions are understood in a pathwise fashion.
Obviously, all three processes are increasing. Moreover, for semimartingales
and $c\searrow0$, under weak non-degeneracy conditions, their values
diverge up to infinity. Thus a natural question arises what the growth
rate of the (upward, downward) truncated variation is. Under a proper
normalization we expect also some convergence to a non-trivial object.
These questions are answered in the following section.

\subsection{Behavior as $c\searrow0$. First order properties for continuous
semimartingales}

For a continuous semimartingale $\cbr X_{t\in[0;T]}$ we will denote
its decomposition by 
\[
X_{t}:=X_{0}+M_{t}+A_{t},\quad t\in[0;T],
\]
where $M$ is a continuous local martingale such that $M_{0}=0$ and
$A$ is a continuous finite variation process such that $A_{0}=0.$
Given $T>0$, by $\mathcal{C}([0;T],\R)$ we denote the usual space
of continuous functions on $[0;T]$ endowed with the topology given
by norm $\norm{\cdot}{\infty}{}$.
\begin{thm}
\label{thm:LLN} Let $T>0$ and let $\cbr X_{t\in[0;T]}$ be a continuous
semimartingale as above. We have 
\[
\lim_{c\searrow0}c\:\TTV Xtc\conv\sb X_{t},\quad\text{a.s.}
\]
\[
\lim_{c\searrow0}c\:\UTV Xtc\conv\sb X_{t}/2,\quad\text{a.s.}
\]
and
\[
\lim_{c\searrow0}c\:\DTV Xtc\conv\sb X_{t}/2,\quad\text{a.s.}
\]
In all cases the converge is understood in the $\mathcal{C}([0;T],\R)$
topology. \end{thm}
\begin{rem}
One can see that $\text{TV}^{c}$ is of order $c^{-1}$. Hence by
the discussion above this is also the lower bound of the total variation
of the approximation of $X$ in $\norm{\cdot}{\infty}{}$-ball of
radius $c/2$. For diffusions we will find finer estimates in the
next section.

Assumptions of \thmref{LLN} could be weakened slightly. Without additional
effort we can prove the theorem for $A$ not being necessary continuous.
This is however cumbersome from notational point of view, as we cannot
work in $\mathcal{C}([0;T],\R)$ space. The problem of non-continuous
semimartingales will be treated in full extent in future papers.
\end{rem}

\begin{rem}
\thmref{LLN} could be considered as some kind of a law of large numbers.
We will now provide a rough justification using the Wiener process
$W$ as an example. One can imagine splitting an interval $[0;1]$
into $c^{-2}$ parts. On each part $W$ performs a motion of order
$c$. The contribution of the part to the total truncated variations
is not negligible and is of order $c$. The contributions are random
and {}``almost'' independent for non-neighboring parts. Therefore
there is no randomness in the limit. 
\end{rem}

\begin{rem}
\label{rem:LLNProof}The heuristics presented in the previous remark
is nice at the intuitive level however a more precise description
is required to perform the proof. In the case of a Wiener process
with drift this will be a precise characterization of $t_{1},t_{2},\ldots$
for which the sup in definition (\ref{eq:TVDefinition}) is attained,
which will lead to a natural renewal structure . In the case of a
general semimartingale following the same path seems to be hopeless.
To circumvent the problem we employed an abstract approach based of
time change techniques in spirit of the Dambis, Dubins-Schwarz theorem
\cite[Chapt. V, Theorem 1.6]{Revuz:1991kx}. 
\end{rem}
Having explained {}``the law of large numbers nature'' of the above
result a natural question arises about the corresponding central limit
theorem. This will be addressed in the next section for $\cbr{X_{t}}_{t\ge0}$
being a diffusion satisfying some mild conditions.

\subsection{Behavior as $c\searrow0$. Second order properties for diffusions}

Let us now consider a general diffusion defined with equation 
\begin{equation}
\dd X_{t}=\sigma(X_{t})\dd W_{t}+\mu(X_{t})\dd t,\quad X_{0}=0,\label{eq:diffusionDef}
\end{equation}
We will always assume that $\sigma,\mu$ are Lipschitz functions and
$\sigma>0$. It is well known, \cite[Sect. IX.2]{Revuz:1991kx},
that under these conditions the equation admits a unique strong solution.
The main result of this section is 
\begin{thm}
\label{thm:CLT}Let $T>0$ then

\begin{multline}
\rbr{X,\UTV Xtc-\frac{1}{2}\rbr{\frac{\sb X_{t}}{c}+X_{t}},\DTV Xtc-\frac{1}{2}\rbr{\frac{\sb X_{t}}{c}-X_{t}},\TTV Xtc-\frac{\sb X_{t}}{c}}\\
\conv^{d}(X,\tilde{M}_{t},\tilde{M}_{t},2\tilde{M}_{t}),\text{ as }c\searrow0,\label{eq:assertion-2}
\end{multline}
where $\tilde{M}$ is given by the change time formula: 
\begin{equation}
\tilde{M}_{t}:=12^{-1/2}B_{\sb X_{t}},\label{eq:limit}
\end{equation}
where $B$ is a standard Brownian motion such that $B$ and $X$ are
independent. The convergence is understood as the weak convergence
in $\mathcal{C}([0;T],\R)^{4}$ topology.\end{thm}
\begin{rem}
Let us notice that by \cite[Proposition 5.33]{Jacod:2003} from the joint convergence of $X$ and three other processes related to $\text{UTV}, \text{DTV}$ and $\text{TV}$ one obtains their stable convergence as described in \cite[Sect. VIII.5]{Jacod:2003}.
\end{rem}
\begin{rem}
Let us now present an intuitive explanation of the result on the example
of a Wiener process with drift, $W$ and the truncated variation.
\thmref{LLN} reads as
\[
c\:\TTV Wtc\conv t,\quad\text{a.s.}
\]
and by \thmref{CLT} and the fact that $\sb W_{t}=t$ we obtain
\[
\TTV Xtc-\frac{t}{c}\conv^{d}3^{-1/2}B_t.
\]
In this case the theorems are indeed an {}``almost classical'' law
of large numbers and central limit theorem. This stems from the fact
that $\text{TV}^{c}$ in this case has a particularly nice, renewal
structure. 

On the intuitive level, by \eqref{variational} one may say that for
any path of $W$ on interval $\ab$, minimal ''vertical'' length
of graph of any random function $f:\left[a;b\right]\rightarrow\mathbb{R},$
uniformly close to this path must be at least equal to 
\[
\frac{b-a}{c}+\sqrt{\frac{b-a}{3}}R_{c},
\]
where $c=2\sup_{t\in\left[a;b\right]}\left|f\left(t\right)-W_{t}\right|,$
and $R_{c}$ is a random variable such that it tends in distribution
to a standard normal distribution $\mathcal{N}\left(0,1\right)$ as
$c\searrow0.$ Note that for small $c$'s this lower bound is almost
deterministic.
\end{rem}

\begin{rem}
It is easy to check that $\sb {\tilde{M}}=\sb X$. Let $M$ be the local
martingale in the semimartingale decomposition of $X$. It is natural
to ask how the laws of $M$ and $\tilde{M}$ are related. The martingale
of the form given by \eqref{limit} were introduced in \cite{Ocone:1993fk}
and are called Ocone martingales. By results of \cite{Vostrikova:2000kx}
it follows that $M$ is an Ocone martingale only if $\sigma=const$
(i.e. $X$ is a Brownian motion with some stochastic drift).

Let us also notice that $\sigma=const$ is also the only case when
$\sb X_{t}$ is a deterministic process. 

Ocone martingales have particularly simple structure which sometimes
makes it easy to draw conclusion about them. As an example we consider
a situation when $\sigma\leq C$. Then 
\[
\pr{\sup_{t\in[0;T]}\tilde{M}_{t}\geq a}\leq\pr{\sup_{t\in[0;CT]}12^{-1/2}B_{t}>a}=\pr{\sup_{t\in[0;T]}B_{t}>(12/C)^{1/2}a},
\]
hence $\tilde{M}$ has a Gaussian concentration. Further properties and references can be found in \cite{Vostrikova:2000kx}. 
\end{rem}

\begin{rem}
The assumption $\sigma>0$ is equivalent to $\sigma\neq0$. This follows
by the fact that $\sigma$ is continuous so, under the assumption
that $\sigma\neq0$, either $\sigma>0$ for any $x$ or $\sigma<0$.
In the latter case one can simply take $-\sigma$ instead of $\sigma$
and obtain a diffusion with the same law. \\
 The case when $\sigma$ may attain value $0$ requires further studies.
To see this let us consider {}``a very degenerate case'' when $\sigma=0$
on an interval $[x_{0};x_{1}]$ for $x_{0}<x_{1}$. For any $x\in(x_{0};x_{1})$
the diffusion degenerates locally to a deterministic process, a solution
of an ordinary differential equation, with a bounded total variation.
Hence the above formulation of the CLT does not make sense. While
this case was relatively easy, the situation becomes more involved
for border points $x_{0},x_{1}$ or {}``isolated'' $0$'s. We suspect
that in such cases a non-trivial correction term containing the local
time may be required. 
\end{rem}

\begin{rem}
\label{rem:CLTproofOutline}Similarly as in the case of the law of
large numbers (see Remark \ref{rem:LLNProof}) the proof splits into
technically different parts. 

The first one deals with the Wiener process with drift $X_{t}=W_{t}+\mu t$.
We use here the fact that $\text{TV}^{c}\left(X,t\right)$ has a fairly
simple renewal-like structure. Moreover, it is possible to derive
explicit formulas for the Laplace transform of the increments of the
truncated variation. Then a very simple argument allows to treat random
drift, i.e. the case where $\mu$ is a random variable independent
of $W$. 

The second step deals with diffusions with $\sigma=const$. Namely,
on a small interval we have $X_{\Delta t+t}-X_{t}\approx\sigma(W_{\Delta t+t}-W_{t})+\mu(X_{t})\Delta t:=Y_{\Delta t}$
which is essentially a Wiener process with a random drift as above.
It turns out that we may control the quality of the approximation
to conclude the proof using some metric-theoretic tricks and the Prohorov
metric in this case.

As explained in Remark \remref{onlySimpleDiffusions}, this approach
fails in the case of non-constant $\sigma$. Here we appeal to a time
change technique and a R\'{e}nyi mixing-like argument (see e.g. \cite[p. 309]{Silvestrov:2004fk}.
A reader familiar with this kind of reasoning may recognize that this
is why we get the independence in \eqref{limit}.
\end{rem}

\subsection{Large time results}

For the Wiener process with drift it is possible to derive results
for large time. In this section, we put
\[
X:=W_{t}+\mu t,
\]
Firstly, we present 
\begin{fact}
\label{fact:TVnas} Let $T>0$ and $c>0$. We have
\[
\lim_{n\conv+\infty}\TTV X{nt}c/n\conv m_{\mu}^{c}t,\quad\text{a.s.,}
\]
 where the convergence is understood in $\mathcal{C}([0;T],\R)$ topology
and
\begin{equation}
m_{\mu}^{c}=\begin{cases}
\mu\coth(c\mu) & \text{if}\:\mu\neq0,\\
c^{-1} & \text{if}\:\mu=0.
\end{cases}\label{eq:mcmu}
\end{equation}
Analogously we have
\[
\lim_{n\conv+\infty}\UTV X{nt}c/n\conv\frac{1}{2}n_{\mu}^{c}t,\quad\text{a.s.},
\]
and
\[
\lim_{n\conv+\infty}\DTV X{nt}c/n\conv\frac{1}{2}n_{-\mu}^{c}t,\quad\text{a.s.}
\]
where again the convergence is understood in $\mathcal{C}([0;T],\R)$ topology
and
\begin{equation}
n_{\mu}^{c}=\begin{cases}
\mu\coth(c\mu) + \mu & \text{if}\:\mu\neq0,\\
c^{-1} & \text{if}\:\mu=0.
\end{cases}\label{eq:ncmu}
\end{equation}
\end{fact}
The quality of the above approximation is studied in 
\begin{thm}
\label{thm:TVTimeRescale}Let $T>0$ and $c>0$. We have 
\[
\frac{\TTV X{nt}c-m_{\mu}^{c}nt}{\sigma_{\mu}^{c}\sqrt{n}}\conv^{d}B_{t},\quad\text{as}\: n\conv+\infty,
\]
 where $\conv^{d}$ is understood as weak convergence in $\mathcal{C}([0;T],\R)$
topology; $m_{\mu}^{c}$ is given by (\ref{eq:mcmu}) and 
\[
\rbr{\sigma_{\mu}^{c}}^{2}=\begin{cases}
\frac{2-2c\mu\coth(c\mu)}{\sinh^{2}(c\mu)}+1 & \text{if}\:\mu\neq0,\\
1/3 & \text{if}\:\mu=0.
\end{cases}
\]

\end{thm}

\begin{thm}
\label{thm:UTVimeRescale}Let $T>0$ and $c>0$. We have 
\[
\frac{\UTV X{nt}c-\frac{1}{2}n_{\mu}^{c}nt}{\rho_{\mu}^{c}\sqrt{n}}\conv^{d}B_{t},\quad\text{as}\: n\conv+\infty,
\]
and 
\[
\frac{\DTV X{nt}c-\frac{1}{2}n_{-\mu}^{c}nt}{\rho_{\mu}^{c}\sqrt{n}}\conv^{d}B_{t},\quad\text{as}\: n\conv+\infty,
\]
 where $\conv^{d}$ is understood as weak convergence in $\mathcal{C}([0;T],\R)$
topology; $n_{\mu}^{c}$ is given by (\ref{eq:ncmu}) and 

\[
\rbr{\rho_{\mu}^{c}}^{2}=\begin{cases}
\frac{2\exp(4c\mu)\rbr{\sinh(2c\mu)-2c\mu}}{\rbr{\exp(2c\mu)-1}^{3}} & \text{if}\:\mu\neq0,\\
1/3 & \text{if}\:\mu=0.
\end{cases}
\]
\end{thm}

\begin{rem}
Fact \ref{fact:TVnas} could be considered as a kind of law of large
numbers. Indeed, $\text{TV}^{c}$ builds up over time (cf.  Subsection \ref{sub:Joint-structure-of}) and because of the homogeneity of $X$ its truncated variation can be decomposed
into a number of independent increments. These increments are also
square integrable, therefore \thmref{TVTimeRescale} and \thmref{UTVimeRescale}
hold. 

The task of proving analogous facts for more general classes of processes
seems to be elusive at the moment. Firstly, our methods failed in
this case, but the reason seems to lie deeper than that. It is connected
with the fact that the truncated variation depends on the paths in
a rather complicated way, simplifying only when $c\searrow0$. We
suspect that it is possible to prove similar results for ergodic Markov
processes. This however seem a little unsatisfactory as in this case
the convergence stems merely from the fact that on distant intervals
the process itself is nearly independent.
\end{rem}

\begin{rem}
It is possible for the finite dimensional distributions of the normalized truncated variation processes appearing in \thmref{TVTimeRescale} and \thmref{UTVimeRescale} to obtain even stronger results, namely the Berry-Ess\'{e}en-type estimates of the rate of convergence to normal distribution. The straightforward way to obtain such estimates is to use the already mentioned cumulative structure of the truncated variation processes of a Brownian motion with drift and \cite[Theorem 8.2]{Roginsky:1994}. One can check that the appropriate moments exist (see formula \eqref{LaplaceTV1} and observe that inter-renewal times in this case have the same distribution as the exit time of Brownian motion with drift from a strip. Thus we obtain that the difference between the cdf of the multidimensional projection of the limit distribution and the cdf of the finite dimensional distributions of the normalized truncated variation processes in \thmref{TVTimeRescale} and \thmref{UTVimeRescale} is of order $\log(n)/\sqrt{n}.$ We suspect that the results of \thmref{CLT} can be strengthened in a similar way. This will be a subject of further studies.
\end{rem}

Let us now comment on the structure of the paper. In the next section
we gather facts about the truncated variation and discuss potential
application to the theory of stochastic processes. Section \ref{sec:ProofCLT}
is devoted to the proof of \thmref{CLT}. In section \ref{sec:ProofLLN}
we present the proof of \thmref{LLN}. Finally in Section \ref{sec:ProofLargeTimes}
we sketch the proof of the large time results presented just above.

\subsection*{Acknowledgments }

We would like to thank Rados\l{}aw Adamczak for help in the proof
of \lemref{measureTh}. We thank also the anonymous referee for useful comments.

\section{Properties of the truncated variation\label{sec:PropertiesTV}}

This section is based on results of \cite{ochowski:2011lr}. For reader's
convenience we keep much of the notation introduced there.

Arguably the most interesting property of the $\text{TV}^{c}$ was
listed in \eqref{variational}. Another closely related property is
given by 
\begin{equation}
\TTV f{\left[a;b\right]}c=\inf\cbr{\TTV g{\left[a;b\right]}{}:g\,\text{such that }\norm{g-f}{osc}{}\leq c,g\left(a\right)=f\left(a\right)},\label{eq:variational2}
\end{equation}
where$\norm h{osc}{}:=\sup\cbr{|h(x)-h(y)|:x,y\in\ab}$. The infimum
in \eqref{variational2} attained for some $g^{0,c}:\ab\mapsto\R$,
which is unique. Moreover, we also have the following explicit representation:
\begin{equation}
g^{0,c}(s)=f(a) + \UTV f{[a;s]}c-\DTV f{[a;s]}c\label{eq:solution2}
\end{equation}
and
\begin{equation}
\norm{g^{0,c}-f}{\infty}{}\leq c.\label{eq:approximationBound}
\end{equation}
$g^{0,c}$ is also closely related to the solution of the problem
stated in \eqref{variational}. Let us put $\alpha_{0}:=-\inf\cbr{g^{0,c}(s)-f(s):s\in\ab}-\frac{1}{2}\norm{g^{0,c}-f}{osc}{}$.
The function $g^{c}$ for which $\inf$ in \eqref{variational} is
attained is given by
\begin{equation}
g^{c}(s):=\alpha_{0}+g^{0,c}(s).\label{eq:solution1}
\end{equation}
The problem posed by \eqref{variational2} seems a little artificial
at first. Its formulation has however a substantial advantage over
the problem of \eqref{variational} when considered in stochastic
setting. Namely, when working with stochastic processes the solution
given by \eqref{solution2} is adaptable to the same filtration as
the process itself while the solution obtained in \eqref{solution1}
requires some {}``knowledge of future''. We would like also to mention
that condition $\norm{g-f}{osc}{}\leq c$ in (\ref{eq:variational2})
implies that the increments of $f$ are uniformly approximated by
the increments of $g^{0,c}$ with accuracy $c$. This property might
be useful for applications to numerical stochastic integration. 

To give the reader some intuition about the functions introduced above
we rephrase \cite[Remark 2.4]{ochowski:2011lr}: {}``$g^{c}$ is
the most lazy function possible, which changes its value only if it
is necessary to stay in the tube defined by $\norm{g^{c}-f}{\infty}{}\leq c/2$''.
This can be seen on the following picture

\begin{figure}[H]
\caption{An example of function (in red) and $g^{c}$ (in various colors).}

\includegraphics[scale=0.45]{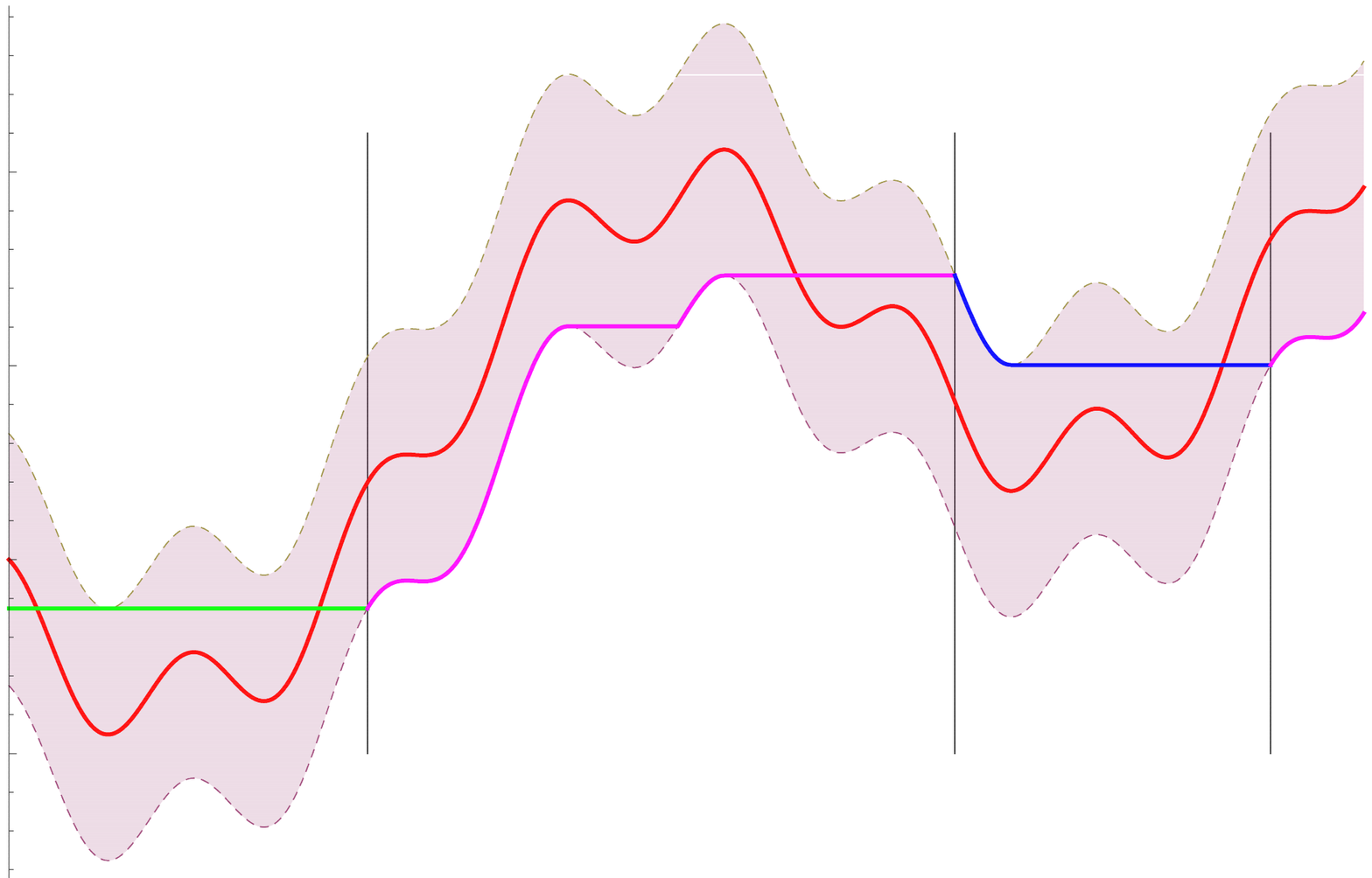}
\end{figure}

We hope that we convince the reader that the truncated variation is
an interesting research object. Moreover, we hope that it will be
useful both in the theory of stochastic processes and in their applications.
The first step towards this goal were undertaken in \cite{ochowski:2011fk}
and \cite{ochowski:2011rs} e.g. in \cite{ochowski:2011fk} was calculated
the Laplace transform of $\text{UTV}^{c}$ and $\text{DTV}^{c}$ for
a Brownian motion with drift and in \cite{ochowski:2011rs}
are presented possible applications to the approximation of stochastic
processes and stochastic integration. 

We plan to report shortly on further findings.

\subsection{Joint structure of $\text{TV}^{c},\text{UTV}^{c}$ and $\text{\text{DTV}}^{c}$.
\label{sub:Joint-structure-of}}

We will now describe the structure of $\text{TV}^{c},\text{UTV}^{c}$
and $\text{\text{DTV}}^{c}$. The construction is described in more
details in \cite[Section 2]{ochowski:2011lr}. Let $-\infty<a<b<+\infty$
and let $f:\left[a;b\right]\rightarrow\mathbb{R}$ be a c\`{a}dl\`{a}g function.
For $c>0$ let us assume that 
\begin{equation}
T_{U}^{c}f:=\inf\left\{ s\geq a:\sup_{t\in\left[a;s\right]}f\left(t\right)-f\left(s\right)\geq c\right\} \leq T_{D}^{c}f:=\inf\left\{ s\geq a:f\left(s\right)-\inf_{t\in\left[a;s\right]}f\left(t\right)\geq c\right\} .\label{eq:firstUpAssumption}
\end{equation}
i.e. the first upward jump of function $f$ of size $c$ appears before
the first downward jump of the same size $c$ or both times are infinite,
i.e. there is no upward or downward jump of size $c.$ Note that when
this condition fails one may simply consider function $-f.$ Now we
define sequences $\left(T_{U,k}^{c}\right)_{k=0}^{\infty},\left(T_{D,k}^{c}\right)_{k=-1}^{\infty},$
in the following way: $T_{D,-1}^{c}=a,$ $T_{U,0}^{c}=T_{U}^{c}f$
and for $k\geq0$:
\[
T_{D,k}^{c}:=\inf\left\{ s\geq T_{U,k}^{c}:\sup_{t\in\left[T_{U,k}^{c};s\right]}f\left(t\right)-f\left(s\right)\geq c\right\} ,
\]
\[
T_{U,k+1}^{c}=\inf\left\{ s\geq T_{D,k}^{c}:f\left(s\right)-\inf_{t\in\left[T_{D,k}^{c};s\right]}f\left(t\right)\geq c\right\} .
\]
Next let us define two sequences of non-decreasing functions $m_{k}^{c}:\left[T_{D,k-1}^{c};T_{U,k}^{c}\right)\rightarrow\mathbb{R}$
and $M_{k}^{c}:\left[T_{U,k}^{c};T_{D,k}^{c}\right)\rightarrow\mathbb{R}$
for $k\geq0$ such that $T_{D,k-1}^{c}<\infty$ and $T_{U,k}^{c}<\infty$
respectively, with the formulas 
\[
m_{k}^{c}\left(s\right):=\inf_{t\in\left[T_{D,k-1}^{c};s\right]}f\left(t\right),\quad M_{k}^{c}\left(s\right)=\sup_{t\in\left[T_{U,k}^{c};s\right]}f\left(t\right),
\]
Similarly, let us define two finite sequences of real numbers $\cbr{m_{k}^{c}}$
and $\cbr{M_{k}^{c}}$, for such $k$'s that $T_{D,k-1}^{c}<\infty$
and $T_{U,k}^{c}<\infty$ by
\begin{equation}
m_{k}^{c}:=m_{k}^{c}\left(T_{U,k}^{c}-\right)=\inf_{t\in\left[T_{D,k-1}^{c};T_{U,k}^{c}\right)}f\left(t\right),\label{eq:defmi}
\end{equation}
\begin{equation}
M_{k}^{c}:=M_{k}^{c}\left(T_{D,k}^{c}-\right)=\sup_{t\in\left[T_{U,k}^{c};T_{D,k}^{c}\right)}f\left(t\right).\label{eq:defMi}
\end{equation}
The above definitions are simple however may be hard to read without
pictures. We hope the following will be helpful. Note that we present
the same function as in the previous example

\begin{figure}[H]
\caption{Example of definition of $T_{U,k}^{c},T_{D,k}^{c}$ and $M_{k}^{c},m_{k}^{c}$.}
\begin{tabular}{cc}
\includegraphics[scale=0.3]{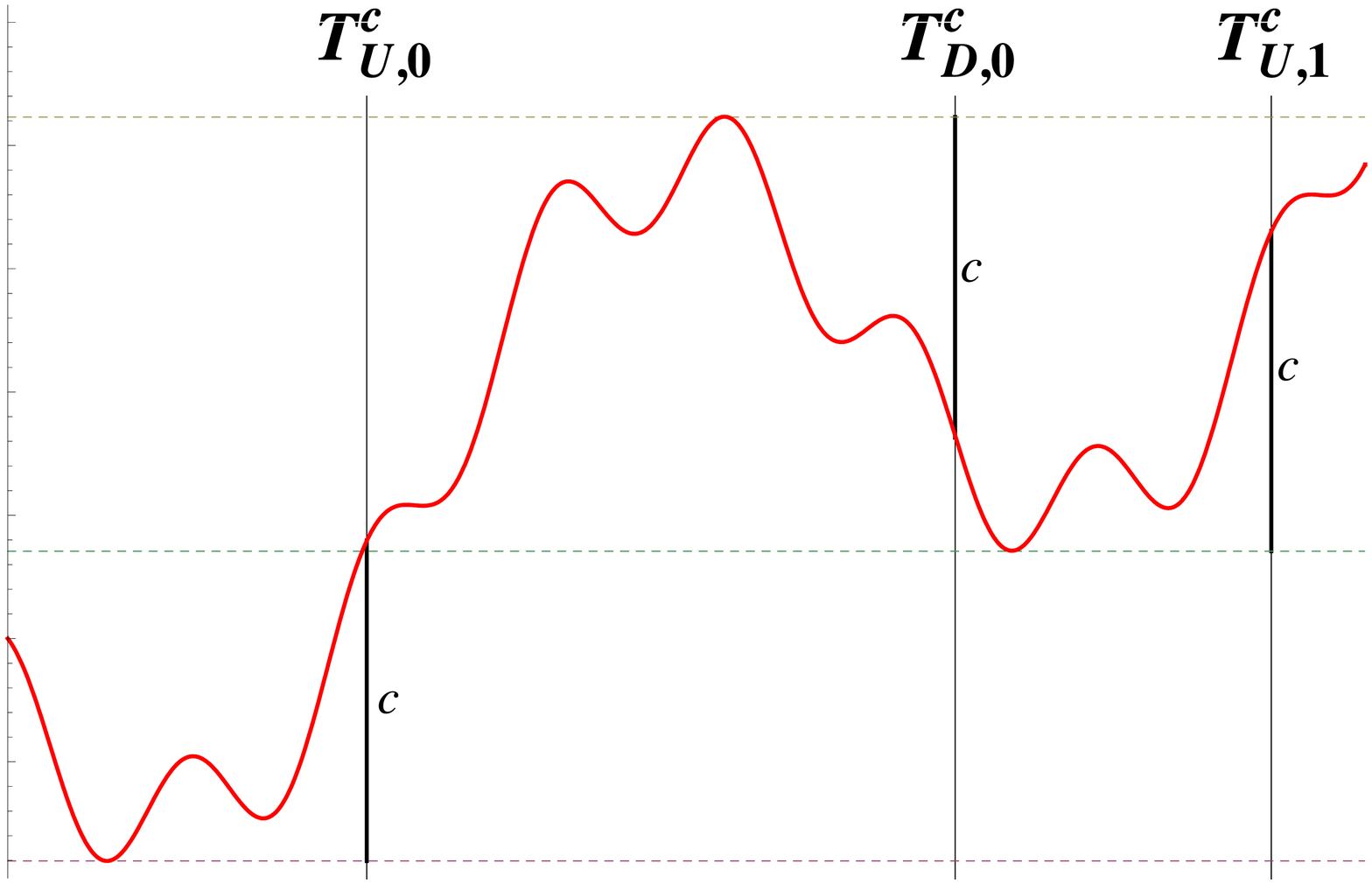} & \includegraphics[scale=0.3]{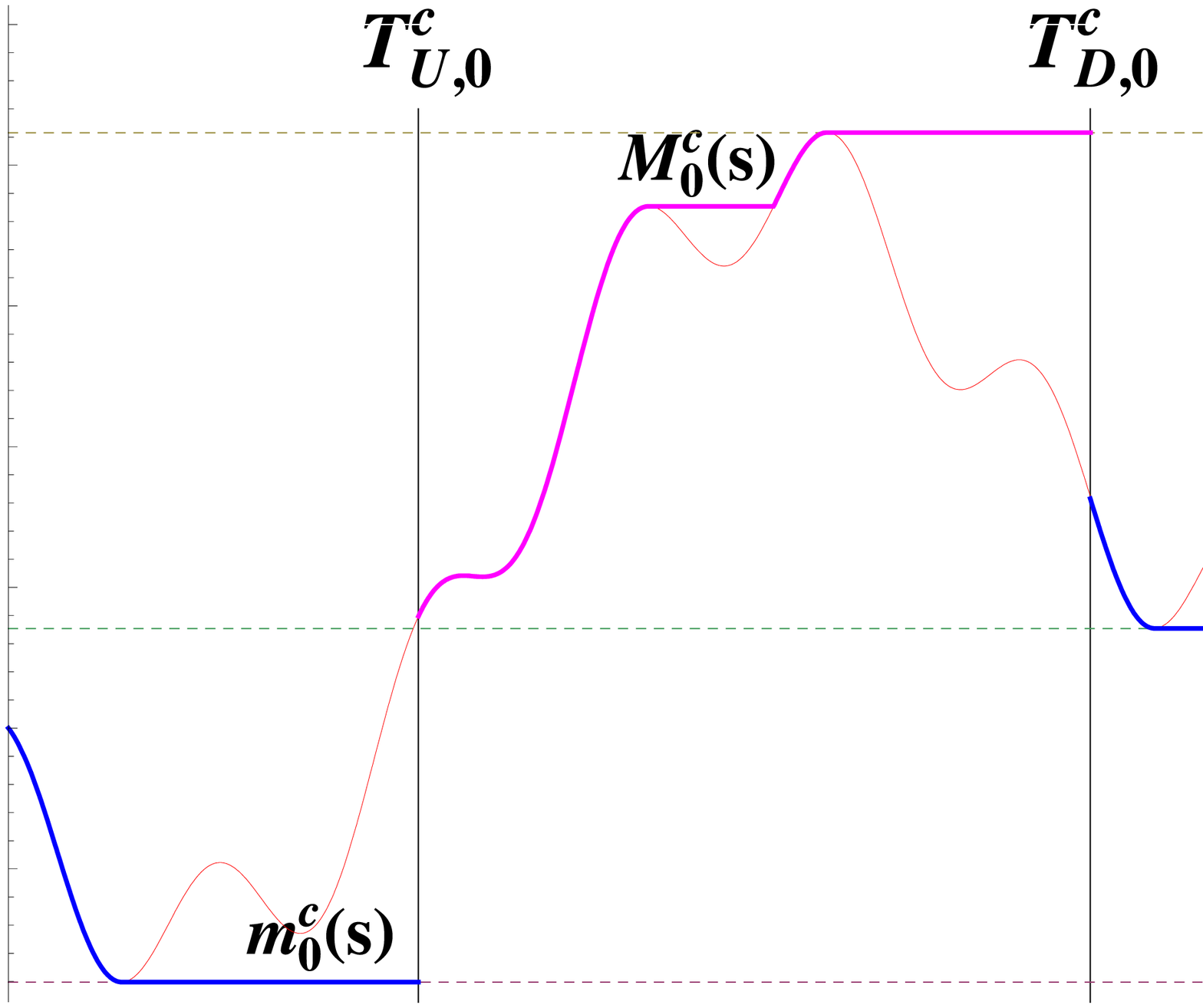}\tabularnewline
\end{tabular}
\end{figure}

The main results of this section is (cf. \cite[Theorem 2.3]{ochowski:2011lr}) 
\begin{thm}
\label{thm:JointStructure}For any c\`{a}dl\`{a}g  function $f:\ab\mapsto\R$
such that \textup{$T_{U}^{c}f\leq T_{D}^{c}f$} we have
\[
\UTV f{[a;s]}c=\DTV f{[a;s]}c=0,
\]
when $s\in\left[a;T_{U,0}^{c}\right)$ and
\[
\UTV f{[a;s]}c:=\begin{cases}
\sum_{i=0}^{k-1}\left\{ M_{i}^{c}-m_{i}^{c}-c\right\} +M_{k}^{c}\left(s\right)-m_{k}^{c}-c & \text{if }s\in\left[T_{U,k}^{c};T_{D,k}^{c}\right),\\
\sum_{i=0}^{k}\left\{ M_{i}^{c}-m_{i}^{c}-c\right\}  & \text{if }s\in\left[T_{D,k}^{c};T_{U,k+1}^{c}\right),
\end{cases}
\]
\[
\DTV f{[a;s]}c:=\begin{cases}
\sum_{i=0}^{k-1}\left\{ M_{i}^{c}-m_{i+1}^{c}-c\right\}  & \text{ if }s\in\left[T_{U,k}^{c};T_{D,k}^{c}\right),\\
\sum_{i=0}^{k-1}\left\{ M_{i}^{c}-m_{i+1}^{c}-c\right\} +M_{k}^{c}-m_{k+1}^{c}\left(s\right)-c & \text{ if }s\in\left[T_{D,k}^{c};T_{U,k+1}^{c}\right).
\end{cases}
\]
Moreover, for any c\`{a}dl\`{a}g  function $f:\ab\mapsto\R$ and any $s\in\ab$
we have 
\begin{equation}
\TTV f{[a;s]}c=\UTV f{[a;s]}c+\DTV f{[a;s]}c.\label{eq:TVisUTVDTV}
\end{equation}

\end{thm}

\subsection{Basic properties of $\TTV f{\ab}c$ , $\UTV f{\ab}c$ and $\DTV f{\ab}c$}

We will now list some properties, most of which is used in the paper.
These are taken from \cite[Section 2.4, Section 2.5]{ochowski:2011lr}.
Unless stated otherwise the functions considered below are c�dl�g 
\begin{itemize}
\item For any strictly increasing and continuous function $s:\R \rightarrow \R$
\begin{equation}
\TTV f{\ab}c=\TTV{f\circ s^{-1}}{[s(a);s(b)]}c,\label{timechange}
\end{equation}
the analogous equalities hold for $\text{UTV}^{c}$ and $\text{DTV}^{c}.$
\item For any $f:\ab\mapsto\R$ and any $c>0$ we have
\begin{equation}
\DTV f{\ab}c=\UTV{-f}{\ab}c.\label{eq:utvDTV}
\end{equation}

\item For any $s\in(a;b)$ we have 
\begin{equation}
\TTV f{\ab}c\geq\TTV f{[a;s]}c+\TTV f{[s;t]}c,\label{inter2}
\end{equation}
 and the analogous inequalities hold for $\text{UTV}^{c}$ and $\text{DTV}^{c}.$
\item On the other hand, for any $s\in(a;b)$ we have 
\begin{equation}
\TTV f{\ab}c\leq\TTV f{[a;s]}c+\TTV f{[s;t]}c+c,\label{inter3}
\end{equation}
and the analogous inequalities hold for $\text{UTV}^{c}$ and $\text{DTV}^{c}.$
\item For any $f,g:\left[a;b\right]\rightarrow\mathbb{R}$ and $c_{1},c_{2}\geq0$
we have
\begin{equation}
\TTV{f+g}{\ab}{c_{1}+c_{2}}\leq\TTV f{\ab}{c_{1}}+\TTV g{\ab}{c_{2}},\label{eq:sum}
\end{equation}
and the analogous inequalities hold for $\text{UTV}^{c}$ and $\text{DTV}^{c}$.
Note that in above we admit some quantities to be infinite in case $c_{1}=0$
or $c_{2}=0$. In particular
\begin{equation}
\left|\TTV{f+g}{\ab}c-\TTV f{\ab}c\right|\leq\TTV g{\ab}{}.\label{eq:lipschitzcondition}
\end{equation}
These facts were not proved in \cite{ochowski:2011lr}. We offer a
proof in Fact \factref{lipschits} below.
\item For any $f:\ab\mapsto\R$ mapping 
\[
(0,+\infty)\ni c\mapsto\TTV f{\ab}c,
\]
is convex and decreasing hence continuous. The same holds true for
$\text{UTV}^{c}$ and $\text{DTV}^{c}$. Moreover, though not mentioned
in \cite{ochowski:2011lr}, it can be easily upgraded to functional
setting. E.g. we define functional $T:(0; +\infty)\mapsto\mathcal{D}$ (Skorohod
space of c\`{a}dl\`{a}g  functions) given by $T(c)(t):=\TTV f{[a;t]}c$ is
convex and decreasing in a point-wise sense.
\item For any $f:\ab\mapsto\R$ we have
\begin{equation}
\lim_{c\searrow0}\TTV f{\ab}c=\TTV f{\ab}{},\label{eq:convergencetoTV}
\end{equation}
we recall that the right-hand side might be infinite. \end{itemize}
\begin{fact}
\label{fact:lipschits}For any $f,g:\left[a;b\right]\rightarrow\mathbb{R}$
and $c_{1},c_{2}\geq0$ we have 
\begin{equation}
\TTV{f+g}{\ab}{c_{1}+c_{2}}\leq\TTV f{\ab}{c_{1}}+\TTV g{\ab}{c_{2}},\label{eq:krupnik}
\end{equation}
and the analogous inequalities hold for $\text{UTV}^{c}$ and $\text{DTV}^{c}.$\end{fact}
\begin{proof}
The inequality for $\text{UTV}^{c}$ holds by definition \eqref{defUTV}
and the inequality
\begin{multline*}
\max\left\{ f\left(s\right)+g\left(s\right)-f\left(t\right)-g\left(t\right)-c_{1}-c_{2},0\right\} =\max\left\{ f\left(s\right)-f\left(t\right)-c_{1}+g\left(s\right)-g\left(t\right)-c_{2},0\right\} \\
\leq\max\left\{ f\left(s\right)-f\left(t\right)-c_{1},0\right\} +\max\left\{ g\left(s\right)-g\left(t\right)-c_{2},0\right\} .
\end{multline*}
By \eqref{utvDTV} we have similar property for $\text{DTV}^{c}$.
Finally, to obtain \eqref{krupnik} it is enough to utilize \eqref{TVisUTVDTV}.
\end{proof}

\section{Proof of Theorem \ref{thm:CLT}\label{sec:ProofCLT}}

The proof structure reflects the outline contained in Remark \ref{rem:CLTproofOutline}.
We start with

\subsection{Proof for Wiener process with drift}

In our proof we will use an Anscombe-like result. It is not much more
than a reformulation of \cite[Theorem 4.5.5]{Silvestrov:2004fk} to
our specific needs. From now on we will use {}``$\cleq$'' to denote
the situation when an equality or inequality holds with some constant
which is irrelevant for calculations. Our setting is as follows. Let
us fix some $T>0$ and 
\[
(D_{i}(c),Z_{i}(c)),\quad i\geq1,
\]
 be sequences of i.i.d. random vectors indexed by certain parameter
$c\in(0,1]$. We define 
\begin{equation}
M_{c}(t):=\min\cbr{i\geq0:\sum_{i=1}^{i+1}D_{i}(c)>t},\label{eq:mcDef}
\end{equation}
 
\begin{equation}
P_{c}(t):=\rbr{\sum_{i=1}^{M_{c}(t)}Z_{i}(c)}-\frac{\ev{Z_{1}(c)}}{\ev{D_{1}(c)}}t,\quad t\in[0;T].\label{eq:pct}
\end{equation}
 Let us observe that such defined $M_{c},P_{c}$ are c\`{a}dl\`{a}g processes.
We will use the following assumptions 
\begin{description}
\item [{(A1)}] For any $c>0$ we have $D_{1}(c)>0$ a.s. and $\ev{D_{1}(c)}\rightarrow0$
as $c\searrow0$.
\item [{(A2)}] We denote $X_{i}(c):=Z_{i}(c)-(\ev{Z_{1}(c)}/\ev{D_{1}(c)})D_{i}(c)$.
We have $\ev{X_{i}(c)}=0$. We assume that there exists $\sigma>0$
such that 
\[
\frac{\ev{X_{1}(c)^{2}}}{\ev{D_{1}(c)}}\conv\sigma^{2},\quad\text{as }c\searrow0.
\]

\item [{(A3)}] There exists $\delta\in(0,2]$ such that
\[
\frac{\ev{|X_{1}(c)|^{2+\delta}}}{\ev{D_{1}(c)}}\conv0,\quad\text{as }c\searrow0.
\]

\item [{(A4)}] There exists $\delta>0,C>0$ such that for any $c\in(0;1]$
we have 
\[
\ev{|D_{1}(c)|^{1+\delta}}\leq C(\ev{D_{1}(c)})^{1+\delta}.
\]

\end{description}
Before formulation of the fact we define 
\begin{equation}
\mathcal{D}:=\mathcal{D}([0;T],\R):=\cbr{f:[0;T]\mapsto\R:f\text{ is c�dl�g}},\label{eq:cadlagSpace}
\end{equation}
we equip this space with $\norm{\cdot}{\infty}{}$-norm. This may
seem unusual, as the Skorohod metric (see \cite[Chapter 3]{Billingsley:1968aa})
is a more natural choice for space $\mathcal{D}$. Let us note however
that in all cases we will obtain the convergence to a continuous limits.
In such case both notions are equivalent (see \cite[Section 18]{Billingsley:1968aa}.
\begin{fact}
\label{fact:AnscombeLike}Let $T>0$ and assume that (A1)-(A4) hold.
Then 
\[
P_{c}\rightarrow^{d}\sigma B,\quad\text{as}\: c\searrow0,
\]
\begin{equation}
(\ev{D_{1}(c)})M_{c}\conv^{d}id,\quad\text{as}\: c\searrow0,\label{eq:mozg}
\end{equation}
where $\sigma^{2}$ is the same as in (A2), $id(x)=x$, and the convergence
is understood as weak convergence in $\mathcal{D}([0;T],\R)$. \end{fact}
\begin{proof}
We define 
\begin{equation}
S_{c}(n):=\sum_{i=1}^{n}Z_{i}(c),\quad V_{c}(n):=\sum_{i=1}^{n}D_{i}(c),\quad n\in\mathbb{N}.\label{eq:defSC}
\end{equation}
 Moreover, let us denote $f(c):=\frac{\ev{Z_{1}(c)}}{\ev{D_{1}(c)}}$
and we recall that $X_{i}(c):=Z_{i}(c)-f(c)D_{i}(c)$. We define a
family of auxiliary processes 
\begin{equation}
P_{c}^{1}(t):=H_{c}(\lfloor g(c)t\rfloor),\quad t\geq0,\label{eq:procss}
\end{equation}
 where $H_{c}(n):=S_{c}(n)-f(c)V_{c}(n)$ and $g(c):=(\ev{D_{1}(c)})^{-1}$.
By (A1) $g(c)\conv+\infty$ as $c\searrow0$.

Now the proof follows by \cite[Theorem 4.5.5, p. 290]{Silvestrov:2004fk}.
The assumptions of \cite[Theorem 4.5.5]{Silvestrov:2004fk} consist of seven conditions denoted
by $\mathcal{T}_{4}\mathcal{,S}_{4},\mathcal{S}_{5}\mathcal{,S}_{7}\mathcal{%
,S}_{8}\mathcal{,S}_{9}$ and $\mathcal{J}_{20}.$ These conditions read as:
\begin{itemize}
\item  $\left( \mathcal{T}_{4}\right)$: $\left( \kappa _{\varepsilon ,k},\xi
_{\varepsilon ,k}\right) ,$ $k=1,2,...,$ is a sequence of i.i.d. random
vectors that take values in $\left[ 0;+\infty \right) \times \mathbb{R;}$

\item  $\left( \mathcal{S}_{4}\right)$: $n_{\varepsilon }\mathbb{P}\left(
\kappa _{\varepsilon ,k}>u\right) \rightarrow \pi _{1}\left( u\right) $ as $%
\varepsilon \rightarrow 0$ for all $u>0,$ which are points of continuity of
the limitting function $\pi _{1}\left( u\right) ;$

\item  $\left( \mathcal{S}_{5}\right)$: $n_{\varepsilon }\mathbb{E}\kappa
_{\varepsilon ,k}1_{\left\{ \kappa _{\varepsilon ,k}\leq u\right\}
}\rightarrow c\left( u\right) $ as $\varepsilon \rightarrow 0$ for some $u>0,
$ which is a point of continuity of $\pi _{1}\left( u\right) ;$

\item  $\left( \mathcal{S}_{7}\right)$: $n_{\varepsilon }\mathbb{P}\left(
\left| \xi _{\varepsilon ,k}\right| >u\right) \rightarrow 0$ as $\varepsilon
\rightarrow 0$ for every $u>0;$

\item  $\left( \mathcal{S}_{8}\right)$: $n_{\varepsilon }\mathbb{E}\left|
\xi _{\varepsilon ,k}\right| 1_{\left\{ \left| \xi _{\varepsilon ,k}\right|
\leq u\right\} }\rightarrow a$ as $\varepsilon \rightarrow 0$ for some $u>0;$

\item  $\left( \mathcal{S}_{9}\right)$: $n_{\varepsilon }\mathbb{D}%
^{2}\left| \xi _{\varepsilon ,k}\right| 1_{\left\{ \left| \xi _{\varepsilon
,k}\right| \leq u\right\} }\rightarrow b^{2}$ as $\varepsilon \rightarrow 0$
for some $u>0;$

\item  $\left( \mathcal{J}_{20}\right)$: $c=c\left( u\right)
-\int_{0}^{u}sd\pi _{1}\left( s\right) >0,$ where $\pi _{1}\left( s\right) $
and $c\left( u\right) $ are obtained in $\left( \mathcal{S}_{4}\right) \ $%
and $\left( \mathcal{S}_{5}\right) $ respectively. 
\end{itemize}
Before verifying assumptions we list how our notation translates to
the one of \cite{Silvestrov:2004fk}, $c$ is $\epsilon$, $\lceil g(c)\rceil$
is $n_{\epsilon}$, $D_{i}(c)$ is $\kappa_{\epsilon,i}$ and $X_{i}(c)$
is $\xi_{\epsilon,i}$. Condition $\mathcal{T}_{4}$ (p. 287) is obviously
fulfilled. Conditions $\mathcal{S}_{4}$ and $\mathcal{S}_{5}$ (p.
283) hold with $\pi(u)=0$ and $c(u)=1$ respectively. Indeed, let
us fix $u>0$. $\mathcal{S}_{4}$ writes as 
\[
\lceil g(c)\rceil\pr{D_{1}(c)>u}\leq\lceil g(c)\rceil u^{-1-\delta}\ev{}|D_{1}(c)|^{1+\delta}\cleq\lceil g(c)\rceil g(c)^{-(1+\delta)}\rightarrow0,\quad\text{ as }c\searrow0,
\]
where we used assumptions (A1), (A4) and the Chebyshev inequality.
We will use a few times an obvious inequality 
\begin{equation}
|x|^{\delta_{1}+\delta_{2}}\geq|u|^{\delta_{1}}|x|^{\delta_{2}},\label{eq:simpleTMP}
\end{equation}
valid for any $\delta_{1},\delta_{2}>0$ and $|x|\geq|u|$. We check
that 
\begin{equation}
\lceil g(c)\rceil\ev{}D_{1}(c)1_{\cbr{D_{1}(c)>u}}\leq\lceil g(c)\rceil u^{-\delta}\ev{}|D_{1}(c)|^{1+\delta}\rightarrow0,\quad\text{ as }c\searrow0,\label{eq:tmp74}
\end{equation}
again by (A1), (A4) and \eqref{simpleTMP}. The expression in condition
$\mathcal{S}_{5}$ writes in our notation as $\lceil g(c)\rceil\ev{}D_{1}(c)1_{\cbr{D_{1}(c)\leq u}}$.
By \eqref{tmp74} its limit does not depend on $u$ and is the same
as the one of 
\[
\lceil g(c)\rceil\ev{}D_{1}(c)\rightarrow1,\quad\text{ as }c\searrow0,
\]
which follows by (A1) and the definition of $g(c)$. 

We will now verify conditions $\mathcal{S}_{7}$, $\mathcal{S}_{8}$,
$\mathcal{S}_{9}$ (p. 287-288) with $a=0$ and $b^{2}=\sigma^{2}$.
Let $u>0$, the condition $\mathcal{S}_{7}$ writes as 
\[
\lceil g(c)\rceil\pr{|X_{1}(c)|\geq u}\leq\lceil g(c)\rceil u^{-(2+\delta)}\ev{}|X_{1}(c)|^{2+\delta}\rightarrow0,\quad\text{ as }c\searrow0,
\]
where we used assumption (A3) and the Chebyshev inequality. Further
we have
\[
\lceil g(c)\rceil\ev{}\rbr{|X_{1}(c)|1_{\cbr{|X_{1}(c)|>u}}}\leq\lceil g(c)\rceil u^{-(1+\delta)}\ev{}|X_{1}(c)|^{2+\delta}\rightarrow0,\quad\text{ as }c\searrow0,
\]
where we used assumption (A3) and \eqref{simpleTMP}. Now $\mathcal{S}_{8}$
follows directly from above and the equality 
\[
\ev{}\rbr{X_{1}(c)1_{\cbr{|X_{1}(c)|>u}}}=-\ev{}\rbr{X_{1}(c)1_{\cbr{|X_{1}(c)|\leq u}}}
\]
 which is a consequence of the fact that $\ev{}X_{i}(c)=0$. Let us
now observe that 
\[
\lceil g(c)\rceil\ev{}\rbr{X_{1}(c)^{2}1_{\cbr{|X_{1}(c)|>u}}}\leq\lceil g(c)\rceil u^{-\delta}\ev{}|X_{1}(c)|^{2+\delta}\rightarrow0,\quad\text{ as }c\searrow0,
\]
where we again used assumption (A3) and \eqref{simpleTMP}. By the
above considerations we have that $\lim_{c\searrow0}\lceil g(c)\rceil\var(X_{1}(c)1_{\cbr{|X_{1}(c)|\leq u}})$
is the same as $\lim_{c\searrow0}\lceil g(c)\rceil\var(X_{1}(c))$.
Now $\mathcal{S}_{9}$ follows directly from (A2). Finally, $\mathcal{J}_{20}$
(p.285) holds with $c=1$ see also \cite[(4.5.2)]{Silvestrov:2004fk}. 

Now, it is straightforward to identify the limit using the description
in \cite[p. 284 and p. 288]{Silvestrov:2004fk}. Indeed, the process
$\kappa_{0}$ (p. 284) is simply given by $\kappa_{0}(t)=t$ (notice
that on the right hand side of formula (4.5.1) in \cite[p. 284]{Silvestrov:2004fk}
one should replace $z$ by $y$) so its inverse $\nu_{0}$ is also
$\nu_{0}(t)=t$ (which proves \eqref{mozg}). The process $\xi_{0}$
is the same as in $\mathcal{A}_{65}$ (p. 288). Let us note that Silvestrov's
$\conv^{U}$ is the same convergence we need, see \cite[Definition 2.4.2]{Silvestrov:2004fk}.

\end{proof}
Let $W$ be a standard Wiener process and $\mu\in\R$. We denote a
Wiener process with drift $\mu$ by 
\begin{equation}
X_{t}:=W_{t}+\mu t,\quad t\geq0.\label{eq:BMwithDrift}
\end{equation}
Our first result is the following
\begin{lem}
\label{lem:BMquadruple}Let $T>0$ and $X$ be a Wiener process with
drift given by (\ref{eq:BMwithDrift}). We have 
\begin{equation}
\rbr{X_{t}-\mu t,\TTV Xtc-\frac{t}{c}}\conv^{d}\rbr{W_{t},3^{-1/2}B_{t}},\quad\text{as }c\searrow0,\label{eq:BrownCLT}
\end{equation}
where $(W,B)$ are independent standard Wiener processes. The convergence
is understood as weak convergence in $\mathcal{C}([0;T],\R)^{2}$
topology.\end{lem}
\begin{proof}
We fix $a,b\in\R$ and define $A_{t}^{c}:=a\TTV Xtc+bX_{t}-\rbr{\frac{a}{c}+b\mu}t$.
Assume that we proved that 
\begin{equation}
A^{c}\conv^{d}\rbr{a^{2}/3+b^{2}}\tilde{B},\quad\text{ as }c\searrow0,\label{eq:tmpConvergence}
\end{equation}
weakly in topology of $\mathcal{C}([0;T],\R)$, where $\tilde{B}$
is some standard Brownian motion. The convergence for $(a,b)=(1,0)$
yields that $\cbr{\TTV Xtc-\frac{t}{c}}_{c>0}$ is tight, hence also
is the sequence of vectors on the left side of \eqref{BrownCLT}.
Now, applying the Cram�r-Wold device \cite[Theorem 7.7]{Billingsley:1968aa}
we easily justify that the convergence of finite-dimensional distributions,
hence (\ref{eq:BrownCLT}) indeed holds.

Now we are to prove (\ref{eq:tmpConvergence}). We transparently transfer
all quantities of Section \ref{sec:PropertiesTV} to the stochastic
setting by applying them in a pathwise fashion, i.e. $f\left(t\right)=X_{t}$.
We denote 
\[
G_{i}(c):=\rbr{M_{i}^{c}-m_{i}^{c}-c}+\rbr{M_{i}^{c}-m_{i+1}^{c}-c},\quad H_{i}(c):=\rbr{M_{i}^{c}-m_{i}^{c}-c}-\rbr{M_{i}^{c}-m_{i+1}^{c}-c}.
\]
By \thmref{JointStructure} and continuity of X we have $\TTV X{T_{U,k}^{c}}c=\sum_{i=0}^{k-1}Y_{i}(c)$
(in fact this holds under additional assumption (\ref{eq:firstUpAssumption})
but this is irrelevant in the limit). By (\ref{eq:approximationBound})
and again by \thmref{JointStructure} we have 
\begin{equation}
\norm{X_{T_{U,k}^{c}}-\sum_{i=0}^{k-1}H_{i}(c)}{\infty}{}\leq c,\quad\text{a.s.}\label{eq:goodApproximation}
\end{equation}
(Note that $X_{0}=0.$) We fix some $a,b\in\R$ and for any $i\geq0$
write
\begin{equation}
Z_{i}(c):=aG_{i}(c)+bH_{i}(c).\label{eq:zi}
\end{equation}
We denote also
\begin{equation}
D_{i}(c):=T_{U,i}^{c}-T_{U,i-1}^{c},\quad i\geq1,\text{ and }\: D_{0}(c):=T_{U,0}^{c}.\label{eq:di}
\end{equation}
The following simple observation will be crucial for the further proof.
Let us notice that by the strong Markov property of $X$ and its space
homogeneity we have that $\cbr{Z_{i}(c)}_{i\geq1}$ and $\cbr{D_{i}(c)}_{i\geq1}$
are i.i.d. sequences. For $i=0$ the distributions are different because
of {}``starting conditions''. The first part, i.e. the values for
$i=0$ disappear in the limit. For notational simplicity from now
on, we will implicitly assume that $i\geq1$. 

We will proceed now in the direction of utilizing Fact \ref{fact:AnscombeLike}.
To do this, we need to calculate moments, fortunately enough \cite{Taylor:1975kx}
provides us with sufficient tools. Using the notation from \cite{Taylor:1975kx}
we may write 
\[
(T_{D,i}^{c}-T_{U,i}^{c},M_{i}^{c}-m_{i}^{c}-c)=^{d}(T_{c},X(T_{c})+c),
\]
where $T_{c},X$ are defined in \cite[Introduction]{Taylor:1975kx}.
Hence the formula \cite[(1.1)]{Taylor:1975kx} reads as 
\begin{equation}
\ev{\exp(\alpha\rbr{M_{i}^{c}-m_{i}^{c}-c}-\beta\rbr{T_{D,i}^{c}-T_{U,i}^{c}})}=\frac{\delta\exp(-(\alpha+\mu)c)\exp(\alpha c)}{\delta\cosh(\delta c)-(\alpha+\mu)\sinh(\delta c)},\label{eq:LaplaceTV1}
\end{equation}
where $\delta=\sqrt{\mu^{2}+2\beta}$. This formula is valid if $\alpha<\delta\coth(\delta c)-\mu$
and $\beta>0$. If $\mu\neq0$ we may also put $\beta=0.$ One may
check that the pair 
\[
\rbr{T_{U,i+1}^{c}-T_{D,i}^{c},M_{i}^{c}-m_{i+1}^{c}-c}
\]
is independent of $(T_{D,i}^{c}-T_{U,i}^{c},M_{i}^{c}-m_{i}^{c}-c)$.
It becomes obvious when one recalls definitions of Section \ref{sec:PropertiesTV}
((\ref{eq:defmi}) and (\ref{eq:defMi}) in particular) and apply
the strong Markov property of $X$. Moreover, we notice that the law
of $\rbr{T_{U,i+1}^{c}-T_{D,i}^{c},M_{i}^{c}-m_{i+1}^{c}-c}$ is the
same as the one of $(T_{D,i}^{c}-T_{U,i}^{c},M_{i}^{c}-m_{i}^{c}-c)$
if we change the drift coefficient to $-\mu$. Therefore, by \cite[(1.1)]{Taylor:1975kx}
we get 
\begin{equation}
\ev{\exp(\alpha\rbr{M_{i}^{c}-m_{i+1}^{c}-c}-\beta\rbr{T_{U,i+1}^{c}-T_{D,i}^{c}})}=\frac{\delta\exp(-(\alpha-\mu)c)\exp(\alpha c)}{\delta\cosh(\delta c)-(\alpha-\mu)\sinh(\delta c)},\label{eq:LaplaceTV2}
\end{equation}
where $\delta=\sqrt{\mu^{2}+2\beta}$ (with the same restrictions
as before). These are enough information to check the moment conditions
required in Fact \ref{fact:AnscombeLike}. Calculations are easy and
straightforward however lengthy. We decided not to include all of
them in the paper. Instead, we list crucial steps and provide the
reader with the Mathematica notebook with all details%
\footnote{\texttt{\url{http://www.mimuw.edu.pl/~pmilos/moments.nb}.} The file can be viewed with a free application available on \texttt{\url{http://www.wolfram.com/products/player/}.}}. 
Combining the above equations and putting $\alpha=0$ (note that
this is always possible for $c$'s small enough) we get
\[
\ev{\exp(-\beta D_{i}(c))}=\frac{2\beta+\mu^{2}}{\beta+\mu^{2}+\beta\cosh\rbr{2c\sqrt{2\beta+\mu^{2}}}}.
\]
Differentiation yields 
\begin{equation}
\ev{D_{i}(c)}=\frac{2\sinh(c\mu)^{2}}{\mu^{2}}=2c^{2}+O(c^{4}).\label{eq:expectanceD}
\end{equation}
One can check that the formula above is valid for $\mu=0$ when we
take the limit. This applies also to the subsequent moments formulae.
Moreover 
\[
\ev{D_{i}(c)^{2}}=\frac{16}{3}c^{4}+O(c^{6}),\quad\ev{D_{i}(c)^{4}}=\frac{7936}{105}c^{8}+O(c^{10}).
\]
This is enough to check conditions (A1) of Fact \ref{fact:AnscombeLike}
as well as (A4) with $\delta=3$. Analogously, by putting $\beta=0$
we calculate that 
\begin{equation}
\ev{\exp(\alpha Z_{i}(c))}=\frac{4\mu^{2}}{\left((a-b)\left(1-e^{-2c\mu}\right)\alpha-2\mu\right)\left((a+b)\left(1-e^{2c\mu}\right)\alpha+2\mu\right)}.\label{eq:laplaceZ}
\end{equation}
Again, by differentiation one gets
\begin{equation}
\ev{Z_{i}(c)}=\frac{2\sinh(c\mu)(a\cosh(c\mu)+b\sinh(c\mu))}{\mu}.\label{eq:expectanceZ}
\end{equation}
And therefore 
\begin{equation}
\frac{\ev{Z_{i}(c)}}{\ev{D_{i}(c)}}=\mu(b+a\coth(c\mu))=\frac{a}{c}+b\mu+O(c).\label{eq:drift}
\end{equation}
Now we have
{\tiny
\begin{multline*}
\ev{\exp\rbr{\alpha Z_{i}(c)-\beta D_{i}(c)}}=\\
\frac{2\left(2\beta+\mu^{2}\right)}{-a^{2}\alpha^{2}+b^{2}\alpha^{2}+2b\alpha\mu+2\left(\beta+\mu^{2}\right)+\left(a^{2}\alpha^{2}+2\beta-b\alpha(b\alpha+2\mu)\right)\cosh(2c\sqrt{2\beta+\mu^{2}})-2a\alpha\sqrt{2\beta+\mu^{2}}\sinh(2c\sqrt{2\beta+\mu^{2}})}.
\end{multline*}
}
Following axiom (A2) we denote $X_{i}(c):=Z_{i}(c)-(\ev{Z_{1}(c)}/\ev{D_{1}(c)})D_{i}(c)$

Using this one may check that 
\[
\ev{X_{i}(c)^{2}}=\frac{3a^{2}-b^{2}-4abc\mu+\left(a^{2}+b^{2}\right)\cosh(2c\mu)-4a^{2}c\mu\coth(c\mu)+2ab\sinh(2c\mu)}{\mu^{2}}.
\]
Now it is straightforward to check (A2) of Fact \ref{fact:AnscombeLike},
viz.
\begin{multline*}
\frac{\ev{X_{i}(c)^{2}}}{\ev{D_{1}}}=\frac{1}{2}\text{cshs}(c\mu){}^{2}\left(3a^{2}-b^{2}-4abc\mu+\left(a^{2}+b^{2}\right)\cosh(2c\mu)\right.\\\left.-4a^{2}c\mu\coth(c\mu)+2ab\sinh(2c\mu\right) =\left(\frac{a^{2}}{3}+b^{2}\right)+\frac{4}{3}abc\mu+O(c^{2}).
\end{multline*}
Finally, one can check that $\ev{X_{i}(c)^{4}}\cleq c^{4}$ and hence
(A3) is verified with $\delta=2$. Having checked all conditions we
conclude that for $P_{c}(t)$ defined by (\ref{eq:pct}) and (\ref{eq:zi}),
(\ref{eq:di}) we have 
\[
P_{c}(t)-\rbr{\frac{a}{c}+b\mu}t\conv^{d}\left(\frac{a^{2}}{3}+b^{2}\right)^{1/2}\tilde{B},\quad\text{as }c\searrow0.
\]
Therefore in order to prove (\ref{eq:tmpConvergence}) it is enough
to to show that $P_{c}(t)-a\TTV Xtc-bX_{t}\conv^{d}0$. By the property
\eqref{goodApproximation} and the continuity of $X$ it follows easily
that it suffices to concentrate on the case $(a,b)=(1,0)$, that is
$A_{t}=\TTV Xtc$. Since $D_{0}(c)$ has different distribution than
$D_{i}(c)$ for $i\geq1$ we introduce two auxiliary objects 
\[
\tilde{M}_{c}(t):=\min\cbr{n\geq0:\sum_{i=0}^{n}D_{i}(c)>t},\quad\tilde{S}_{c}(n)=\sum_{i=0}^{n}Z_{i}(c),
\]
and 
\[
\tilde{P}_{c}(t):=\tilde{S}_{c}(\tilde{M}_{c}(t)).
\]
This differs slightly from $P_{c}$, however, one easily checks that
$\tilde{P}_{c}-P_{c}\conv^{d}0$. By \thmref{JointStructure} we see
that the processes $\TTV Xtc$ and $\tilde{S}_{c}\rbr{\tilde{M}_{c}(t)}$
coincide at random times $T_{U,i}^{c}$, $i\geq1$ moreover, both
are increasing, hence, for any $T\geq0$ and $\varepsilon>0$ 
\[
\pr{\sup_{t\in\left[0,T\right]}\left|\TTV Xtc-\tilde{S}_{c}(\tilde{M}_{c}(t))\right|>\varepsilon}\leq\pr{\sup_{t\in[0;T]}Z_{\tilde{M}_{c}(t)}(c)>\varepsilon}.
\]
Using this we estimate 
\begin{multline*}
\pr{\sup_{t\in\left[0,T\right]}\left|TV_{c}^{\mu}\left(t\right)-\tilde{S}_{c}\rbr{\tilde{M}_{c}(t)}\right|>\varepsilon}\\
\leq\pr{\max_{k\leq2T/\ev D_{1}\left(c\right)+1}Z_{k}\left(c\right)\geq\varepsilon}+\pr{\tilde{M}_{c}\left(T\right)\geq\frac{2T}{\ev D_{1}\left(c\right)}+1}.
\end{multline*}
 The first term could be estimated by the Chebyshev inequality and
the estimates of $\ev{Z_{1}(c)^{4}}$ and $\ev{D_{1}(c)}$ 
\[
\pr{\max_{k\leq2T/\ev{D_{1}(c)}+1}|Z_{i}(c)|>\varepsilon}\leq\rbr{\frac{2T}{\ev{D_{1}(c)}}+1}\frac{\ev{Z_{1}(c)^{4}}}{\varepsilon^{4}}\conv0,\quad\text{as}\: c\conv0.
\]
 The convergence of the second term to $0$ could be established by
Fact \ref{fact:AnscombeLike}.
\end{proof}

\subsection{Proof for diffusions with $\sigma=const$}

We start with a yet simpler case. Namely, let $W$ be a standard Brownian
motion and $X$ be a random variable. Let us define process $Z$ by
\[
Z_{t}:=W_{t}+Xt,\quad t\geq0.
\]

\begin{lem}
\label{lem:StochasticDrift}Let $T>0$. Let us assume that that $W$
and X are independent then 
\[
\rbr{X,W,\TTV Ztc-\frac{t}{c}}\conv^{d}\rbr{X,W,3^{-1/2}B},\text{ as }c\searrow0.
\]
where $B$ is a standard Brownian motion and $X,W,B$ are independent.
The convergence is understood in weak sense in the product topology
of $\R\times\crt^{2}$.\end{lem}
\begin{proof}
We will proceed by the very definition of the weak convergence. Let
$f:\R\times\mathcal{C}([0;T],\R)^{2}\mapsto\R$ be a bounded continuous
function. We have 
\begin{multline*}
\lim_{c\searrow0}\ev{}f\rbr{X,W,\TTV Ztc-\frac{t}{c}}=\lim_{c\searrow0}\ev{}\ev{}\rbr{\left.f\rbr{x,W,\TTV Ztc-\frac{t}{c}}\right|X=x}\\
=\ev{\lim_{c\searrow0}\ev{}\rbr{\left.f\rbr{x,W,\TTV Ztc-\frac{t}{c}}\right|X=x}}\\=\ev{\ev{}\rbr{\left.f\rbr{x,W,3^{-1/2}B}\right|X=x}}=\ev{}f\rbr{X,W,3^{-1/2}B}.
\end{multline*}
where we used \lemref{BMquadruple} and the Lebesgue dominated convergence
theorem. 
\end{proof}
We will deal now with diffusion given by an equation 
\begin{equation}
dX_{t}=\dd W_{t}+\mu(X_{t})\dd t,\quad X_{0}=0,\label{eq:simplifiedDiffusion}
\end{equation}
 i.e. we set $\sigma\equiv1$ in \eqref{diffusionDef}. We assume
also that $\mu$ is bounded and Lipschitz. This process is essentially
a Brownian motion with {}``a variable drift''. We denote
\begin{equation}
\mu^{*}=\sup_{x\in\R}|\mu(x)|<+\infty.\label{eq:mubound}
\end{equation}
We will us the discretion technique. To this end we need to be able
to control the increments of $X$. The following simple lemma is the
first, most crude step of our analysis 
\begin{lem}
\label{lem:CrudeEstimate}Let $t\geq0$ and $\delta>0$ then for any
$b>0$ we have 
\[
\pr{\sup_{s\in[t;t+\delta]}|X_{s}-X_{t}|\geq(\mu^{*}+b)\delta}\leq2\exp\rbr{-b^{2}\delta/2}.
\]
\end{lem}
\begin{proof}
We know that 
\[
X_{t}=X_{0}+W_{t}+\int_{0}^{t}\mu(X_{s})\dd s.
\]
Hence, we have $X_{s}-X_{t}\in(W_{s}-W_{t}-\mu^{*}(s-t),W_{s}-W_{t}+\mu^{*}(s-t))$.
Now the lemma follows by \cite[Proposition II.1.8]{Revuz:1991kx}. 
\end{proof}
Let us fix $T>0,n=1,2,...$ and denote $t_{i}^{n}:=i\frac{T}{n},i\in\cbr{0,1,\ldots,n}$.
We define the {}``approximated'' truncated variation process by
\begin{equation}
ATV^{n,c}(t):=\sum_{i=0}^{\lfloor nt\rfloor-1}\TTV X{[t_{i}^{n};t_{i+1}^{n}]}c+\TTV X{[t_{\lfloor nt\rfloor}^{n};t]}c,\label{eq:atv}
\end{equation}
Its name is justified by 
\begin{lem}
\label{lem:variationsSum}We have

\[
ATV^{n,c}(t)-\TTV Xtc\conv0,\:\text{a.s.}\quad\text{when }c\searrow0,
\]
 and the convergence is understood in $\mathcal{C}([0;T],\R)$ topology.\end{lem}
\begin{proof}
By (\ref{inter2}) one easily verifies that $ATV^{n,c}(t)\leq\TTV Xtc$.
On the other hand, by (\ref{inter3}), $\TTV Xtc-ATV^{n,c}(t)\leq nc$,
for any $t\in[0;T]$. 
\end{proof}
We will take now a detour of the main flow of the proof in order to
collect weak convergence facts used below. First we recall the Prokhorov
metric. Let $(S,d)$ be a metric space and $\mathcal{P}(S)$ be
the space of Borel probability measures on $S$. We topologise $\mathcal{P}(S)$
with the Prokhorov metric
\begin{equation}
d_{P}(P,Q):=\inf\cbr{\epsilon>0:P(F)\leq Q(F^{\epsilon})+\epsilon,\quad\text{for all closed }F\subset S},\label{eq:prohorovMetric}
\end{equation}
in the above expression $F^{\epsilon}:=\cbr{x\in S:\inf_{y\in F}d(x,y)<\epsilon}.$
It is well-known that when $(S,d)$ is separable then convergence
with respect to $\pm{\cdot,\cdot}$ is equivalent to weak convergence.
We refer the reader to \cite[Chapter 3]{Ethier:1986uq} and \cite[Theorem 3.3.1]{Ethier:1986uq}
in particular. Given two random variables $X,Y$ with values in the
same space we will write 
\[
\pm{X,Y}:=\pm{\mathcal{L}(X),\mathcal{L}(Y)},
\]
where $\mathcal{L}(X)$ denotes the law of $X$.

In some parts of our analysis we will need the space of c�dl�g functions
$\mathcal{D}([0;T],\R)$ introduced by (\ref{eq:cadlagSpace}). We
will also use the following product space
\begin{equation}
\mathcal{C}\times\mathcal{D}:=\crt\times\mathcal{D}([0;T],\R),\label{eq:CDspace}
\end{equation}
always with the norm given by $\norm{(f,g)}{}{}:=\norm f{\infty}{}+\norm g{\infty}{}$.
\begin{lem}
\label{lem:TVEstimate}Let $(X,Y)$ be random variables with values
in $\mathcal{C}\times\mathcal{D}$, moreover let $A$ be an event.
Then
\[
\pm{(X,Y),(X1_{A},Y)}\leq2(1-\pr A).
\]
\end{lem}
\begin{proof}
It is enough to apply \cite[Theorem 3.1.2]{Ethier:1986uq} with $\mu=\mathcal{L}((X,Y),(X1_{A},Y))$.\end{proof}
\begin{lem}
\label{lem:CballProch}Let $X:=(X_{1},X_{2})$ and $Y:=(Y_{1},Y_{2})$
be a random variable with values in $\mathcal{C}\times\mathcal{D}$
such that 
\[
\pr{\norm{X_{1}-Y_{1}}{\infty}{}\geq\epsilon/2}\leq\epsilon/2\:\text{ and }\:\pr{\norm{X_{2}-Y_{2}}{\infty}{}\geq\epsilon/2}\leq\epsilon/2,
\]
then 
\[
\pm{X,Y}\leq\epsilon.
\]
\end{lem}
\begin{proof}
We calculate 
\[
\pr{\norm{(X_{1}-Y_{1},X_{2}-Y_{2})}{}{}\geq\epsilon}\leq\pr{\norm{X_{1}-Y_{1}}{}{}\geq\epsilon/2}+\pr{\norm{X_{2}-Y_{2}}{}{}\geq\epsilon}\leq\epsilon,
\]
now the proof follows directly by application of \cite[Theorem 3.1.2]{Ethier:1986uq}.
\end{proof}
We are ready to prove the main result of this part of the proof which
is an upgrade of Lemma \ref{lem:BMquadruple} to {}``simplified diffusions''
given by \eqref{simplifiedDiffusion}. 
\begin{fact}
\label{fact:SimpleMainThm}Let $T>0$. We have 
\begin{equation}
\rbr{X,\TTV Xtc-\frac{t}{c}}\conv^{d}(X,3{}^{-1/2}B),\text{ as }c\searrow0,\label{eq:simpleDiffusionConv}
\end{equation}
where the convergence is understood as weak convergence in $\mathcal{C}([0;T],\R^{d})^{2}$
topology and $B$ is a Brownian motion independent of $X$.\end{fact}
\begin{proof}
We recall that $t_{i}^{n}:=\frac{i}{n}T$, fix some $A\geq\mu^{*}+1$
and define random sets 
\[
A_{i}^{n}:=[X_{t_{i}^{n}}-A/n^{1/4};X_{t_{i}^{n}}+A/n^{1/4}].
\]
We also define random variables 
\[
\mu_{i}^{n}:=\mu(X_{t_{i}^{n}}).
\]
and events
\[
E_{i}^{n}:=\cbr{X_{s}\in A_{i}^{n},\text{ for }s\in[t_{i}^{n};t_{i+1}^{n}]},\quad E^{n}:=\bigcap_{i\in\cbr{0,1,\ldots,n-1}}E_{i}^{n}.
\]
 Using \lemref{CrudeEstimate} we check that for $n$ large enough
we have $\pr{E_{i}^{n}}\geq1-2n\exp(-n^{1/2}/2)$. Consequently, $\pr{E^{n}}\rightarrow1$
as $n\conv+\infty$. For $n\in\mathbb{N}$ we define c�dl�g processes
$\cbr{X_{t}^{n}}_{t\in[0;T]}$ which approximate our diffusion: 
\[
X_{t}^{n}:=X_{t_{i}^{n}}+\mu_{i}^{n}(t-t_{i}^{n})+W_{t}-W_{t_{i}^{n}},\quad\text{whenever }t\in[t_{i}^{n};t_{i+1}^{n}).
\]
One easily checks that $X^{n}\rightarrow X$ a.s. with respect to
$\norm{\cdot}{\infty}{}$. Let us recall (\ref{eq:atv}), we
define its counterpart for $X^{n}$, viz.,

\[
H^{n,c}(t):=\sum_{i=0}^{\lfloor nt\rfloor-1}\TTV{X^{n}}{[t_{i}^{n};t_{i+1}^{n}]}c+\TTV{X^{n}}{[t_{\lfloor nt\rfloor}^{n};t]}c.
\]
One checks (using the same method as in the proof of Lemma \ref{lem:variationsSum})
that 
\[
H^{n,c}(t)-\TTV{X^{n}}tc\conv0,\:\text{a.s.}\quad\text{when }c\searrow0,
\]
norm $\norm{\cdot}{\infty}{}$. On each interval $t\in[t_{i}^{n};t_{i+1}^{n})$
we have 
\begin{equation}
X_{t}^{n}-X_{t}=\int_{t_{i}^{n}}^{t}(\mu_{i}^{n}-\mu(X_{s}))\dd s.\label{eq:difference}
\end{equation}
We observe that conditionally on $E_{i}^{n}$ this expression defines
a function of $t$ which is Lipschitz with constant $w_{n}\leq Ln^{-1/4}$
for some $L>0$. This follows by the fact that $\mu$ is a Lipschitz
function itself. By (\ref{eq:sum}) applied with $c_{1}=c$ and $c_{2}=0$,
conditionally on $E_{i}^{n}$, we have that
\[
\TTV{X^{n}}{[t_{i}^{n},t]}c-w_{n}(t-t_{i}^{n})\leq\TTV X{[t_{i}^{n},t]}c\leq\TTV{X^{n}}{[t_{i}^{n},t]}c+w_{n}(t-t_{i}^{n}),
\]
for any $t\in[t_{i}^{n};t_{i+1}^{n}]$. Further 
\begin{equation}
1_{E^{n}}H^{n,c}(t)-w_{n}T\leq1_{E^{n}}ATV^{n,c}(t)\leq1_{E^{n}}H^{n,c}(t)+w_{n}T,\label{eq:hahahaKrowa}
\end{equation}
for any $t\in[0;T]$. In other words: $\norm{1_{E^{n}}ATV^{n,c}(t)-1_{E^{n}}H^{n,c}(t)}{\infty}{}\leq2w_{n}T$.
Lemma \ref{lem:CballProch} implies that 
\begin{equation}
\pm{(X^{n},1_{E^{n}}ATV^{n,c}),(X^{n},1_{E^{n}}H^{n,c})}\leq4w_{n}T.\label{eq:tmp17}
\end{equation}
It will be crucial that this estimate is uniform in $c$. Let us denote
$L^{n,c}:=\rbr{H^{n,c}(t)-c/t}$. Lemma \ref{lem:StochasticDrift}
applied term by term to $H^{n,c}$ yields the functional convergence
\begin{equation}
(L^{n,c},X^{n})\rightarrow^{d}(3^{-1/2}B,X^{n}),\quad\text{ as }c\searrow0,\label{eq:tmpConv}
\end{equation}
where $B$ and $X^{n}$ are independent. In the above, we understand
the convergence as the functional one in $\mathcal{C}\times\mathcal{D}$
(see also (\ref{eq:CDspace}))

The rest of the proof will follow by a metric-theoretic considerations.
Let us denote 
\begin{alignat*}{2}
 & X_{1}(c):=\rbr{\TTV Xtc-t/c,X}, & X_{2}(c,n):=\rbr{\TTV Xtc-t/c,X^{n}},\\
 & X_{3}(c,n):=\rbr{ATV^{n,c}(t)-t/c,X^{n}}, & X_{4}(c,n):=\rbr{1_{E^{n}}(ATV^{n,c}(t)-t/c),X^{n}},\\
 & X_{5}(c,n):=\rbr{1_{E^{n}}(H^{n,c}(t)-t/c),X^{n}}, & X_{6}(c,n):=\rbr{H^{n,c}(t)-t/c,X^{n}},\\
 & X_{7}(n):=\rbr{3^{-1/2}B,X^{n}}, & X_{8}:=\rbr{3^{-1/2}B,X}.
\end{alignat*}
Let us fix some $\epsilon>0$. We find $n_{1,2}$ such that for any
$n\geq n_{1,2}$ we have $\pm{X_{1}(c),X_{2}(c,n)}\leq\epsilon$ which
is possible by \lemref{CballProch} and convergence $X^{n}\conv X$
.We find $n_{3,4}$ such that for any $n\geq n_{3,4}$ we have $\pm{X_{3}(c,n),X_{4}(c,n)}\leq\epsilon$
which is possible by \lemref{TVEstimate} and estimation of the probability
of $E^{n}$. Further we find $n_{4,5}$ such that for any $n\geq n_{4,5}$
we have $\pm{X_{4}(c,n),X_{5}(c,n)}\leq\epsilon$ which is given by
\eqref{tmp17}. Next, we check that for any $n\geq n_{3,4}$ we have
$\pm{X_{5}(c,n),X_{6}(c,n)}\leq\epsilon$ as well. Finally, we choose
$n_{7,8}$ such that for any $n\geq n_{7,8}$ we have $\pm{X_{7}(n),X_{8}}\leq\epsilon$
which holds by \lemref{CballProch}. We denote $N=\max(n_{1,2},n_{3,4},n_{4,5},n_{7,8})$,
obviously for this $N$ all the above inequalities hold simultaneously
for any $c>0$.

Now we choose $c_{0}$ such that for any $c\leq c_{0}$ we have $\pm{X_{2}(c,N),X_{3}(c,N)}\leq\epsilon$
and $\pm{X_{6}(c,N),X_{7}(c,N)}\leq\epsilon$. The first one is possible
by \lemref{variationsSum} and \lemref{CballProch} and the second
one by \eqref{tmpConv} and again \lemref{CballProch}. Using the
triangle inequality multiple times one obtains
\[
\pm{X_{1}(c),X_{8}}\leq8\epsilon,\quad\text{for any }c\leq c_{0},
\]
This yields convergence \eqref{simpleDiffusionConv} since $\epsilon$
was arbitrary.\end{proof}
\begin{rem}
\label{rem:onlySimpleDiffusions}We strongly believe that it is not
possible to improve the above proof to general diffusions. The main
reason is that without $\sigma=const$ assumption equation \eqref{difference}
is not longer true. Consequently, the estimate in \eqref{tmp17} does
not depend only on $w_{n}$ but also on $c$. Even worse, one can
check that the estimate diverges to infinity as $c\searrow0$. We
could change $n$ and $c$ simultaneously in a smart way so that the
estimate is still useful. However a new problem emerges then, namely
estimate in \lemref{variationsSum} also depend on $n$ and $c$.
It appears that it is not possible to change $n$ and $c$ is such
way that both estimates converge to $0$ when $c\searrow0$. 
\end{rem}

\subsection{Proof for general diffusion}

Now we proceed to the general case. Before proving Theorem \ref{thm:CLT}
we present some measure-theoretic considerations. In the reasoning
below by $\mathbb{W}$ we denote the Wiener measure on $\crt$, see
e.g. \cite[Proposition I.3.3]{Revuz:1991kx}, and by $H$ we denote
the Cameron-Martin space, see \cite[Definition VIII.2.1]{Revuz:1991kx}.
Moreover by $\mathcal{H}$ we denote algebra (i.e. class closed under
finite sums and finite intersections) generated by open balls with
centers in $H$. We have 
\begin{lem}
\label{lem:measureTh}Let $h:\crt\mapsto\R_{+}$ be a measurable mapping
such that $\int h(f)\mathbb{W}(\dd f)=1$. Then for any $\epsilon>0$
there exists $m\in\mathbb{N}$, sets $A_{1},A_{2},\ldots,A_{m}\in\mathcal{H}$
and $h_{1},h_{2},\ldots,h_{m}\in\R_{+}$ such that 
\begin{equation}
\int_{\mathcal{C}}|h_{\epsilon}(f)-h(f)|\mathbb{W}(\dd f)\leq\epsilon,\label{eq:equationHorriblis}
\end{equation}
where 
\[
h_{\epsilon}(f):=\sum_{i=1}^{m}h_{i}1_{A_{i}}(f).
\]
Moreover, one may choose such $A_{1},A_{2},\ldots,A_{m}$ that for
all $i\leq m$, $\mathbb{W}(\partial A_{i})=0$. \end{lem}
\begin{proof}
In the proof we will write $\mathcal{C}$ instead of $\crt$ and $B(f,r)$
will denote an open ball with convention $B(f,0)=\emptyset$. Let
us notice that without loss of generality we can assume that $h$
is bounded by some $l>0$ and has compact support, say contained in
ball $B(0,R)$. Indeed for any function $h$ and any $\epsilon>0$
we can choose $l,R$ such that $\int_{\mathcal{C}}|h(f)1_{\cbr{h\leq l}}1_{\cbr{f\in B(0,R)}}-f(f)|\mathbb{W}(f)\leq\epsilon/2$.
Now it is enough to approximate $h(f)1_{\cbr{h\leq l}}1_{\cbr{f\in B(0,R)}}$
with accuracy $\epsilon/2$. Therefore from now on we will work implicitly
with the assumptions listed above. 

Let us denote
\[
S_{k}:=\cbr{f\in\mathcal{C}:h(f)\in\left(k\epsilon/2,(k+1)\epsilon/2\right]},
\]
for $k\in\cbr{0,1,\ldots,2l/\epsilon}$. We note that by our assumption
sets $S_{k}$ are bounded. We put $\delta:=\epsilon^{2}/(4l^{2})$.
By the regularity of $\mathbb{W}$ (see \cite[Theorem 1.1.1]{Billingsley:1968aa})
we can find open sets $O_{k}$ such that
\begin{equation}
S_{k}\subset O_{k}\quad\text{and}\quad\mathbb{W}(O_{k}\setminus S_{k})\leq\delta/2.\label{eq:setEstimates}
\end{equation}
It is well known that $\mathcal{C}$ is a separable space and $H$
is its dense subspace so one can easily find a countable subset $\cbr{f_{1},f_{2},\ldots}\subset H$
which is dense in $\mathcal{C}$. For each $f_{i}$ we define $r_{i}^{k}:=\sup\cbr{r:B(f_{i},r)\subset O_{k}}/2$
(by convention we put $r_{i}^{k}:=0$ if the set is empty). One promptly
proves that $O_{k}=\bigcup_{i}B(f_{i},r_{i}^{k})$. By the continuity
of measure there exists $i_{k}\in\mathbb{N}$ such that 
\[
\mathbb{W}(O_{k})-\mathbb{W}(\bigcup_{i\leq i_{k}}B(f_{i},r_{i}^{k}))\leq\delta/2.
\]
 Let us denote $A_{k}:=\bigcup_{i\leq i_{k}}B(f_{i},r_{i}^{k})$.
We now define 
\[
h_{\epsilon}(f):=\sum_{k}\rbr{\frac{k}{2}\epsilon}1_{A_{k}}(f).
\]
We will now show that $h_{\epsilon}$ is a good approximating function.
We recall that by the construction $A_{k}\subset O_{k}$ and $W(O_{k}\setminus A_{k})\leq\delta/2$.
This together with \eqref{setEstimates} yields that $\mathbb{W}(S_{k}\Delta A_{k})\leq\delta$,
where $\Delta$ denotes the symmetric difference. We have
\begin{multline*}
\int_{\mathcal{C}}|h_{\epsilon}(f)-h(f)|\mathbb{W}(\dd f)=\int_{\mathcal{C}}\left|\sum_{k}\rbr{k\epsilon/2}1_{A_{k}}(f)-\sum_{k}h(f)1_{S_{k}}(f)\right|\mathbb{W}(\dd f)\\
\leq\sum_{k}\int_{\mathcal{C}}\left|(k\epsilon/2)1_{A_{k}}(f)-h(f)1_{S_{k}}(f)\right|\mathbb{W}(\dd f)\\
\leq\frac{\epsilon}{2}\sum_{k}\mathbb{W}(A_{k}\cap S_{k})+l\sum_{k}\mathbb{W}(S_{k}\Delta A_{k})\leq\frac{\epsilon}{2}+l\frac{2l}{\epsilon}\delta=\epsilon.
\end{multline*}

To check $\mathbb{W}(\partial A_{k})=0$ is is enough to prove that
for any $f\in H$ and any $r>0$ we have $\mathbb{W}(\partial B(f,r))=0$.
By \cite[Theorem VIII.2.2]{Revuz:1991kx} it is enough to show that
$\mathbb{W}(\partial B(0,r))=0.$ This holds by the fact that $\sup$
of the Wiener process has a continuous density (see \cite[Section III.3]{Revuz:1991kx}). 
\end{proof}
Finally we present 
\begin{proof}
(of \thmref{CLT}). We first will show that in order to prove \eqref{assertion-2}
it is enough to prove 
\begin{equation}
\rbr{X,\TTV Xtc-\frac{\sb{X_{t}}}{c}}\conv^{d}\rbr{X,2M},\quad\text{as }c\searrow0,\label{eq:tmp99}
\end{equation}
where $M$ is the same as in Theorem \ref{thm:CLT}. Since $X_{0}=0$
by \eqref{approximationBound} there exists process $\cbr{R_{c}(t)}_{t\in[0;T]}$
such that $\norm{R_{c}}{\infty}{}\leq c$ almost surely and 
\begin{equation}
X_{t}=\UTV Xtc-\DTV Xtc+R_{c}(t).\label{eq:wrona}
\end{equation}
This together with \eqref{TVisUTVDTV} yields that 
\begin{equation}
\UTV Xtc=\frac{1}{2}\rbr{\TTV Xtc+X_{t}-R_{c}(t)},\label{eq:utvfromTV}
\end{equation}
Therefore 
\[
\UTV Xtc-\frac{1}{2}\rbr{\frac{\sb{X_{t}}}{c}+X_{t}}=\frac{1}{2}\rbr{\TTV Xtc-\frac{\sb{X_{t}}}{c}}-\frac{1}{2}R_{c}(t).
\]
Now the convergence follows simply by fact that $\TTV Xtc-\sb{X_{t}}/c$
is a continuous transformation of \eqref{tmp99} and by \cite[Corollary 2, p.31]{Billingsley:1968aa},
\cite[Theorem 4.1]{Billingsley:1968aa}. A completely analogous argument
proves the convergence of $\DTV Xtc-\frac{1}{2}\rbr{\sb{X_{t}}/c-X_{t}}$.
The joint convergence in \eqref{assertion-2} can be established in
the same way.

It will be more convenient to work with additional assumption that
\begin{equation}
C_{1}\geq\sigma\geq C_{2}>0,\label{eq:additionalAss}
\end{equation}
for some constants $C_{1},C_{2}>0$. At the end of the proof we will
remove this assumption. Diffusion \eqref{diffusionDef} writes in
the integral form as 
\[
X_{t}=\int_{0}^{t}\sigma(X_{s})\dd W_{s}+\int_{0}^{t}\mu(X_{s})\dd s.
\]
Let us define $\beta_{t}:=\int_{0}^{t}\sigma(X_{s})^{2}\dd s=\sb{X_{t}}$,
its inverse $\alpha_{t}:=\inf\cbr{s\geq0:\beta_{s}>t}$ and
\[
\tilde{X}_{t}:=X_{\alpha_{t}},\quad t\in[0;T_{0}],\:\text{where }T_{0}:=C_{2}^{2}T.
\]
By the time-change formula \cite[Theorem 8.5.7]{Oksendal:2003kx}
we obtain that $\tilde{X}$ is also a diffusion fulfilling equation
\[
\tilde{X_{t}}=\tilde{W_{t}}+\int_{0}^{t}\frac{\mu(\tilde{X}_{s})}{\sigma^{2}(\tilde{X}_{s})}\dd s,
\]
for some Brownian motion $\tilde{W}$. We chose such $T_{0}$ that
the definition is valid (i.e. $\alpha_{T_{0}}\leq T$). We note also
that $x\mapsto\frac{\mu(x)}{\sigma^{2}(x)}$ is a Lipschitz function.
Let us now denote the natural filtration of $\tilde{X}$ (and $\tilde{W}$)
by $\mathcal{F}$. Making the reverse change of time we get $X_{t}=\tilde{X}_{\beta_{t}}$.
We denote also $\mathcal{G}_{t}:=\mathcal{F}_{\beta_{t}}$. Now we
can apply Fact \ref{fact:SimpleMainThm}. We know that 
\begin{equation}
\rbr{\text{CTV}^{c}(\tilde{X},t),\tilde{X}}\conv^{d}(B,\tilde{X}),\label{eq:convHaha}
\end{equation}
where $CTV^{c}\left(X,t\right):=TV^{c}\left(X,t\right)-\frac{c}{t}$
and $B$ and $\tilde{X}$ are independent. Let us also note that $\text{CTV}^{c}$
can be regarded as a measurable mapping $\text{CTV}^{c}:\crt\mapsto\crt$.

Now, let $K\in\mathcal{H}$ be non-empty set. We check that the measure
$\pr{\tilde{X}\in\cdot|X\in K}$ is absolutely continuous with respect
to $\pr{\tilde{X}\in\cdot}$. Indeed one needs only to check that
$\pr{X\in K}>0$. By the Radon-Nikod�m theorem \cite[Theorem A.1.3]{Kallenberg:1997fk}
there exists a measurable function $h$ such that 
\begin{equation}
\pr{\tilde{X}\in\dd f|X\in K}=h(f)\pr{\tilde{X}\in\dd f}.\label{eq:radonNikodym}
\end{equation}
Using this fact we can leverage \eqref{convHaha}. Let us first note
that by the portmanteau theorem \cite[Theorem I.2.1]{Billingsley:1968aa}
and \cite[Theorem I.2.2]{Billingsley:1968aa} and standard topological
considerations we know that \eqref{convHaha} is equivalent to
\begin{equation}
\pr{\cbr{CTV^{c}(\tilde{X})\in K_{1}}\cap\cbr{\tilde{X}\in K_{2}}}\rightarrow\pr{B\in K_{1}}\pr{\tilde{X}\in K_{2}},\quad\forall_{K_{1},K_{2}\in\mathcal{H}}.\label{eq:setsConvergence}
\end{equation}
Further, by \eqref{radonNikodym}, we have 
\begin{multline*}
a_{c}:=\pr{\cbr{CTV^{c}(\tilde{X})\in K_{1}}\cap\cbr{\tilde{X}\in K_{2}}\cap\cbr{X\in K}}=\pr{X\in K}\\\pr{\cbr{CTV^{c}(\tilde{X})\in K_{1}}\cap\cbr{\tilde{X}\in K_{2}}|X\in K}\\
=\pr{X\in K}\int_{\mathcal{C}}h(f)1_{\cbr{CTV^{c}(f)\in K_{1}}}1_{\cbr{f\in K_{2}}}\pr{\tilde{X}\in\dd f},
\end{multline*}
where $\pr{\tilde{X}\in\dd f}$ is the same as the Wiener measure.
We now approximate $h$ with accuracy $\epsilon=1/n$ with simple
function $h^{n}$ satisfying conditions of Lemma, \ref{lem:measureTh}
(we use additional superscript $^{n}$ to denote the case we are referring
to). Hence we have
\begin{equation}
\int_{\mathcal{C}}|h^{n}(f)-h(f)|\pr{\tilde{X}\in\dd f}\leq\frac{1}{n},\label{eq:lab}
\end{equation}
 We define 
\[
a_{c}^{n}:=\pr{X\in K}\int_{\mathcal{C}}h^{n}(f)1_{\cbr{CTV^{c}(f)\in K_{1}}}1_{\cbr{f\in K_{2}}}\pr{\tilde{X}\in\dd f}.
\]
 One easily checks that for any $c>0$ there is $|a_{c}-a_{c}^{n}|\leq1/n$.
Applying \eqref{setsConvergence} we obtain 
\begin{multline*}
a_{c}^{n}=\pr{X\in K}\sum_{i=1}^{m^{n}}h_{i}^{n}\pr{\cbr{CTV^{c}(\tilde{X})\in K_{1}}\cap\cbr{\tilde{X}\in A_{i}^{n}\cap K_{2}}}\\
\rightarrow_{c\searrow0}\pr{B\in K_{1}}\pr{X\in K}\sum_{i=1}^{m^{n}}h_{i}^{n}\pr{\tilde{X}\in A_{i}^{n}\cap K_{2}}\\
=\pr{B\in K_{1}}\pr{X\in K}\int_{K_{2}}h^{n}(f)\pr{\tilde{X}\in\dd f}=:a^{n}.
\end{multline*}
It is easy to check that $|a^{n}-a|\leq1/n$, where 
\[
a:=\pr{B\in K_{1}}\pr{X\in K}\int_{K_{2}}h(f)\pr{\tilde{X}\in\dd f}=\pr{B\in K_{1}}\pr{\cbr{\tilde{X}\in K_{2}}\cap\cbr{X\in K}}.
\]
Using the standards arguments we obtain that for any $K_{1},K_{2},K\in\mathcal{H}$
\begin{multline*}
\pr{\cbr{CTV^{c}(\tilde{X})\in K_{1}}\cap\cbr{\tilde{X}\in K_{2}}\cap\cbr{X\in K}}\\
\rightarrow_{c\searrow0}\pr{B\in K_{1}}\pr{X\in K}\int_{K_{2}}h(f)\pr{\tilde{X}\in\dd f}=\pr{B\in K_{1}}\pr{\cbr{\tilde{X}\in K_{2}}\cap\cbr{X\in K}}.
\end{multline*}
Using \cite[Theorem I.2.2]{Billingsley:1968aa} in the same spirit
as in the case of \eqref{setsConvergence} we get
\[
(CTV^{c}(\tilde{X}),\tilde{X},X)\conv^{d}(B,\tilde{X},X),\quad\text{ as }c\searrow0.
\]
where $B$ is independent of $(\tilde{X},X)$ hence also of $\beta$.
Changing the time according to this process we obtain 
\[
\TTV{\tilde{X}}{\beta_{t}}c-\frac{\beta_{t}}{c}\conv^{d}B_{\beta_{t}},
\]
This equation is well-defined as long as $t\leq T_{00}=T_{0}/C_{1}^{2}=TC_{2}^{2}/C_{1}^{2}$.
Our final step is to use (\ref{timechange}) in order to get
\[
\TTV Xtc-\frac{\sb X_{t}}{c}\conv^{d}B_{t}.
\]
So far we have obtained convergence in the space $\mathcal{C}([0;T_{00}],\R)$.
Taking the initial value of $T$ larger (which is possible as our
diffusion is well defined on the whole line) we can obtain the convergence
in $\mathcal{C}([0;T],\R).$

We are yet to remove assumption (\ref{eq:additionalAss}). For any
$N>0$ we put 
\[
\sigma^{N}(x):=\begin{cases}
\sigma(x), & \text{ if }|x|\leq N,\\
\sigma(N), & \text{ if }x>N,\\
\sigma(-N), & \text{ if }x<-N,
\end{cases}\quad\mu^{N}(x):=\begin{cases}
\mu(x), & \text{ if }|x|\leq N,\\
\mu(N), & \text{ if }x>N,\\
\mu(-N), & \text{ if }x<-N.
\end{cases}
\]
 We define a family of diffusions by 
\[
\dd X_{t}^{N}:=\sigma^{N}(X_{t}^{N})\dd W_{t}+\mu(X_{t}^{N})\dd t,\quad X_{0}^{N}=0.
\]
We assume that this diffusion is driven by the same $W$ as in (\ref{eq:diffusionDef})
and that $X,X^{N}$ are coupled in such a way that $X_{t}^{N}=X_{t}$
and $\sb{X_{t}^{N}}=\sb{X_{t}}$ whenever $t\leq\tau^{N}:=\inf\cbr{t\geq0:|X_{t}^{N}|>N}=\inf\cbr{t\geq0:|X_{t}|>N}$.
The solution of (\ref{eq:diffusionDef}) is a continuous process and
exists on the whole line, therefore for any $T>0$ we have 
\[
1_{\cbr{\tau^{N}\leq T}}\conv_{N\conv+\infty}0,\text{ a.s. }
\]
We notice now that $X^{N}$ fulfills (\ref{eq:additionalAss}), hence
the thesis of Theorem \ref{thm:CLT} is already proved for it. The
quantities studied in the proof are equal for $X^{N}$ and $X$ on
the set $\cbr{\tau^{N}\leq T}$. Using the metric-theoretic arguments
as in the proof of Fact \ref{fact:SimpleMainThm} one easily concludes
the proof.
\end{proof}

\section{Proof of Theorem \ref{thm:LLN}\label{sec:ProofLLN}}

As indicated in Introduction the proof splits into two parts. In the
first one we will prove Theorem \ref{thm:LLN} in the case when $X$
is a Wiener process with a drift. This will serve as a key step for
the second part of the proof in which, using time change techniques
we will elevate the result to a general class of semimartingales.

\subsection{Proof for Wiener process with drift\label{sub:LLNSmallCWiner}}

This is much simpler compared to the proof of Lemma \ref{lem:BMquadruple},
therefore we provide only a sketch leaving details to the reader.
Let $X$ be a Wiener process with drift, i. e.
\[
X_{t}:=W_{t}+\mu t,
\]
for a standard Wiener process $W$ and $\mu\in\R$. We have
\begin{lem}
\label{lem:LLNBM}Let $T>0$ and let $\cbr X_{t\in[0;T]}$ be a Wiener
process with drift. Then 
\[
\lim_{c\searrow0}c\:\TTV Xtc\conv\sb X_{t},\quad\text{a.s.}
\]
The converge is understood in the $\mathcal{C}([0;T],\R)$ topology. \end{lem}
\begin{proof}
Firstly, we recall $S_{c}(n)$ defined in \eqref{defSC} and $Z_{i}(c)$
given by \eqref{zi}. We want to show that process $X_{t}(c):=cS_{c}(\lceil g(c)t\rceil)$
converges to a linear function. Let us consider 
\[
M_{n}(c):=c\sum_{i=1}^{n}(Z_{i}(c)-\ev{Z_{i}(c)}).
\]
It is a centered martingale. Differentiation of \eqref{laplaceZ}
yields that (we have $a=1,b=0$ in this case)
\[
\ev{(Z_{i}(c)-\ev{Z_{i}(c)})^{2}}=\frac{2\cosh(2c\mu)\sinh(c\mu)^{2}}{\mu^{2}}=2c^{2}+O(c^{3}).
\]
Therefore, by the Doob inequality and \eqref{expectanceD} we have
\[
\ev{}\rbr{\sup_{t\in[0;T]}\left[cS_{c}(\lceil g(c)t\rceil)-c\lceil g(c)t\rceil\ev{Z_{i}(c)})\right]^{2}}\leq LTc^{2},
\]
for some constant $L$. Using \eqref{expectanceZ} and \eqref{drift}
one obtains 
\[
X_{t}(c)\conv id,\quad\text{a.s.,}
\]
where $id(t)=t$ and the converge holds in $\crt$ topology. Now one
proves an analogous convergence for process $V_{c}(\lceil g(c)t\rceil)$,
which is roughly speaking the inverse of $M_{c}$. To finish the proof
one needs to argue similarly as in the second part of the proof of
\lemref{BMquadruple}.
\end{proof}

\subsection{Proof for semimartingales}

Now we assume that $X_{t}=X_{0}+M_{t}+A_{t},$where $M$ is a continuous
local martingale and $A$ is a process with bounded variation.

To avoid notational inconveniences we assume that $M,A$ are defined
on $[0;+\infty)$ (one can simply put a constant process after $T$).
Let us now introduce an additional standard Brownian motion $\beta$
independent of $X$ and denote
\[
X^{\epsilon}:=X+\epsilon\beta,\quad\epsilon>0.
\]
This is a simple trick to avoid the case when $\sb X$ is not strictly
increasing. Indeed we have $\sb{X_{t}^{\epsilon}}=\sb{X_{t}}+\epsilon t$.
Obviously this is a strictly increasing function. Moreover its inverse,
denoted by $\alpha$, is almost surely Lipschitz with constant smaller
then $\epsilon^{-1}$. Let us denote also $M^{\epsilon}:=M+\epsilon\beta$.
The DDS theorem \cite[Theorem V.1.6]{Revuz:1991kx} ensure that there
exists a Brownian motion $B$ such that 
\[
M_{t}^{\epsilon}:=B_{\sb{X_{t}^{\epsilon}}},\quad t\in[0;T].
\]
Using (\ref{timechange}) we have
\begin{equation}
c\:\TTV{X^{\epsilon}}tc=c\:\TTV{B_{t}+A_{\alpha_{t}}}{\sb{X^{\epsilon}}_{t}}c.\label{eq:tmp1}
\end{equation}
Let us fix $N>0$. Applying (\ref{eq:lipschitzcondition}) to the
paths of $B_{t}+A_{\alpha_{t}}$ we get 
\[
|\TTV{B_{t}+A_{\alpha_{t}}}{\sb{X_{t}^{\epsilon}}\wedge N}c-\TTV{B_{t}}{\sb{X_{t}^{\epsilon}}\wedge N}c|\leq\TTV{A_{\alpha_{t}}}{\sb{X^{\epsilon}}_{t}\wedge N}{}.
\]
Using \lemref{LLNBM} and the above estimate one gets that
\[
c\:\TTV{B_{t}+A_{\alpha_{t}}}{\sb{X_{t}^{\epsilon}}\wedge N}c\conv_{c\searrow0}\sb{X_{t}^{\epsilon}}\wedge N\quad\text{a.s.}
\]
where convergence is understood in $\mathcal{C}([0;T],\R)$ topology.
Moreover, the limit agrees with the limit of \eqref{tmp1} on the
set $\cbr{\sb{X^{\epsilon}}\leq N}$. Hence we obtain
\[
c\:\TTV{X^{\epsilon}}tc\conv\sb{X^{\epsilon}},\quad\text{a.s.}
\]
Our aim now is to get rid of $\epsilon$. We fix some $\alpha\in(0,1)$
and notice that by (\ref{eq:sum}) we have 
\[
c\:\TTV Xtc\leq c\:\TTV{X^{\epsilon}}t{\alpha c}+c\:\TTV{\epsilon B}t{(1-\alpha)c}.
\]
Therefore we have
\[
\limsup_{c\searrow0}c\:\TTV Xtc\leq\alpha^{-1}\sb{X_{t}^{\epsilon}}+(1-\alpha)^{-1}\epsilon^{2}t=\alpha^{-1}\sb{X_{t}}+\rbr{(1-\alpha)^{-1}+\alpha^{-1}}\epsilon^{2}t.
\]
By converging $\epsilon\conv0$ and $\alpha\rightarrow1$ one can
obtain
\[
\limsup_{c\searrow0}c\:\TTV Xtc\leq\sb{X_{t}},\quad\text{a.s.}
\]
Analogously one obtains a lower-bound for $\liminf$. Therefore we
proved that for any $t>0$ we have 
\[
\lim_{c\searrow0}c\:\TTV Xtc=\sb{X_{t}},\quad\text{a.s.}
\]
This is a one dimensional convergence but one easily extends it to
the finite dimensional one. Moreover, since the trajectories are almost
surely increasing the finite dimensional convergence can be upgraded
to the functional one. This follows by the simple fact that if $f_{n}\in\crt$
is a sequence of continuous increasing functions converging point-wise
to a continuous function then the convergence is in fact uniform.

In order to prove the convergence for $\text{UTV}^{c}$ it suffices
to use \eqref{utvfromTV}. $\text{DTV}^{c}$ follows similarly.

\section{Proof of large times results\label{sec:ProofLargeTimes}}

In this section we will only prove \thmref{TVTimeRescale}. It follows
by a similar argument as in the proof of \lemref{BMquadruple}. This
time $c$ is fixed and $n$ will go to infinity. The analogues of
\eqref{zi} and \eqref{di} are given by
\[
Z_{i}(n):=n^{-1/2}Y_{i}(c).
\]
and
\[
D_{i}(n):=n^{-1}(T_{U,i}^{c}-T_{U,i-1}^{c}),\quad i\geq1,\text{ and }\: D_{0}(n):=n^{-1}T_{U,0}^{c}.
\]
By \eqref{drift} (with $a=n^{1/2},b=0$) we have
\[
\frac{\ev{Z_{i}(n)}}{\ev{D_{i}(n)}}=n^{-1/2}\mu\coth(c\mu).
\]
We define $X_{i}(n):=Z_{i}(n)-(\ev{Z_{i}(n)}/\ev{D_{i}(n)})D_{i}(n)$.
Repeating calculations as in the proof of \lemref{BMquadruple} one
obtains
\[
\text{Var}(X_{i}(n))=\frac{3+\cosh(2c\mu)-4c\mu\coth(c\mu)}{n\mu^{2}}.
\]
One checks that $\text{Var}(X_{i}(n))/\ev{D_{i}(n)}=\rbr{\sigma_{\mu}^{c}}^{2}$
as in \thmref{TVTimeRescale}. Now in order to obtain this theorem
it is enough to apply Fact \factref{AnscombeLike}. This is an easy
task. Above we already checked (A1), (A2) and (A3) are trivial. (A4)
holds with any $\delta>0$. 

We skip the proof of \factref{TVnas} which is a simpler version of
the proof in Section \ref{sub:LLNSmallCWiner}. Proofs of \thmref{UTVimeRescale}
follows similarly to the one above with an exception that $Z_{i}(n)=M_{i}^{c}-m_{i}^{c}-c$
in the case of $\text{UTV}^{c}$ and $Z_{i}(n)=M_{i}^{c}-m_{i+1}^{c}-c$
in the case of $\text{DTV}^{c}$.

\bibliographystyle{plain}
\bibliography{branching}

\begin{thebibliography}{10}

\bibitem{Billingsley:1968aa}
Patric Billingsley.
\newblock {\em Convergence of Probability Measures}.
\newblock John Wiley, New York, 1968.

\bibitem{Ethier:1986uq}
Stewart~N. Ethier and Thomas~G. Kurtz.
\newblock {\em Markov processes: Characterization and Convergence}.
\newblock Wiley Series in Probability and Mathematical Statistics: Probability
  and Mathematical Statistics. John Wiley \& Sons Inc., New York, 1986.

\bibitem{Jacod:2003}
Jean Jacod and Albert~N. Shiryaev.
\newblock {\em Limit Theorems for stochastic Processes}, volume 288 of {\em
  Grundlehren der Mathematischen Wissenschaften [Fundamental Principles of
  Mathematical Sciences]}.
\newblock Springer, Berlin, Heidelberg, New York, 2003.

\bibitem{Kallenberg:1997fk}
Olav Kallenberg.
\newblock {\em Foundations of Modern Probability}.
\newblock Probability and its Applications. Springer, 1997.

\bibitem{Lochowski:2008}
R.~{\L}ochowski.
\newblock On truncated variation of {B}rownian motion with drift.
\newblock {\em Bull. Pol. Acad. Sci. Math.}, 56(3-4):267--281, 2008.

\bibitem{ochowski:2011lr}
R.~{\L}ochowski.
\newblock On pathwise uniform approximation of processes with cadlag
  trajectories by processes with minimal total variation.
\newblock {\em submitted to Ann. Inst. Henri Poincar\'{e} Probab. Stat.}, 2011.

\bibitem{ochowski:2011rs}
R.~{\L}ochowski.
\newblock On two applications of truncated variation.
\newblock {\em eprint arXiv:1106.2630}, 2011.

\bibitem{ochowski:2011fk}
R.~{\L}ochowski.
\newblock Truncated variation, upward truncated variation and downward
  truncated variation of brownian motion with drift - their characteristics and
  applications.
\newblock {\em Stoch, Proc. Appl.}, 121:378--393, 2011.

\bibitem{Ocone:1993fk}
D.~L. Ocone.
\newblock A symmetry characterization of conditionally independent increment
  martingales.
\newblock In {\em Barcelona {S}eminar on {S}tochastic {A}nalysis ({S}t.\
  {F}eliu de {G}u\'\i xols, 1991)}, volume~32 of {\em Progr. Probab.}, pages
  147--167. Birkh\"auser, Basel, 1993.

\bibitem{Oksendal:2003kx}
Bernt {\O}ksendal.
\newblock {\em Stochastic differential equations}.
\newblock Universitext. Springer-Verlag, Berlin, sixth edition, 2003.
\newblock An introduction with applications.

\bibitem{Revuz:1991kx}
Daniel Revuz and Marc Yor.
\newblock {\em Continuous martingales and {B}rownian motion}, volume 293 of
  {\em Grundlehren der Mathematischen Wissenschaften [Fundamental Principles of
  Mathematical Sciences]}.
\newblock Springer-Verlag, Berlin, 1991.

\bibitem{Roginsky:1994}
A.~Roginsky.
\newblock A central limit theorem for cumulative processes.
\newblock {\em Adv. Appl. Prob.}, 26:104--121, 1994.

\bibitem{Silvestrov:2004fk}
Dmitrii~S. Silvestrov.
\newblock {\em Limit theorems for randomly stopped stochastic processes}.
\newblock Probability and its Applications (New York). Springer-Verlag London
  Ltd., London, 2004.

\bibitem{Taylor:1975kx}
H.~M. Taylor.
\newblock A stopped {B}rownian motion formula.
\newblock {\em Ann. Probability}, 3:234--246, 1975.

\bibitem{Vostrikova:2000kx}
L.~Vostrikova and M.~Yor.
\newblock Some invariance properties (of the laws) of {O}cone's martingales.
\newblock In {\em S\'eminaire de {P}robabilit\'es, {XXXIV}}, volume 1729 of
  {\em Lecture Notes in Math.}, pages 417--431. Springer, Berlin, 2000.

\end{thebibliography}

\end{document}